\documentclass[a4paper]{amsart}

\usepackage{amsthm,amsmath,amssymb}
\usepackage{mathrsfs}
\usepackage[numbers]{natbib}  
\usepackage[colorlinks,linkcolor=black,citecolor=blue,urlcolor=blue]{hyperref}  

\usepackage{enumerate}

\DeclareMathAlphabet{\mathcal}{OMS}{cmsy}{m}{n}

\usepackage[numbers]{natbib}  

\usepackage{enumerate}

\usepackage{pgf}
\usepackage{tikz}
\usepackage{url}
\usetikzlibrary{arrows,decorations.pathmorphing,backgrounds,positioning,fit,petri}
\usepackage{tikz}
\usetikzlibrary{automata}
\usetikzlibrary{graphs}
\usepackage{rotating}
\usetikzlibrary{shadows}

\newtheorem*{thma}{Theorem~A}
\newtheorem*{thmb}{Theorem~B}

\newtheorem*{cora}{Corollary~A}

\newtheorem{thm}{Theorem}[section]
\newtheorem{lemma}[thm]{Lemma}
\newtheorem{cor}[thm]{Corollary}

\newtheorem{fact}[thm]{Fact}
\newtheorem{claim}[thm]{Claim}
\newtheorem{subclaim}[thm]{Subclaim}

\newtheorem{corollary}[thm]{Corollary} 
\theoremstyle{definition}
\newtheorem{definition}[thm]{Definition}

\newtheorem*{convention}{Convention} 

\theoremstyle{remark}
\newtheorem{remark}[thm]{Remark}  
\DeclareMathOperator{\otp}{otp}
\DeclareMathOperator{\crit}{crit}
\DeclareMathOperator{\rud}{rud}
\DeclareMathOperator{\pred}{pred}

\newcommand{\cM}{\mathcal{M}}
\newcommand{\cN}{\mathcal{N}}
\newcommand{\cQ}{\mathcal{Q}}
\newcommand{\Uu}{\mathcal{U}}
\newcommand{\Ult}{Ult}
\newcommand{\lh}{lh}

\newcommand{\dirlim}{dirlim}
\newcommand{\ran}{ran}

\newcommand{\cW}{\mathcal{W}}

\DeclareMathOperator{\h}{ht}

\newcommand{\ord}{OR}
\newcommand{\gch}{GCH}
\newcommand{\zf}{ZF}
\newcommand{\zfc}{ZFC}
\newcommand{\Tt}{\mathcal{T}}

\newcommand{\cK}{\mathcal{K}}

\DeclareMathOperator{\cf}{cf}
\DeclareMathOperator{\dom}{dom}

\makeindex
\makeindex
\title{Tall Cardinals in Extender Models }

\author{Gabriel Fernandes$^{*}$}
\address{Department of Mathematics, Bar-Ilan University, Ramat-Gan 5290002, Israel.}
\urladdr{http://u.math.biu.ac.il/$\sim$zanettg}
\email{zanettg@macs.biu.ac.il}
\thanks{$^{*}$The author is funded by the European Research Council (grant
	agreement ERC-2018-StG 802756) as a postdoctoral fellow at Bar-Ilan
	University.}
\author{Ralf Schindler$^{\ddagger}$}
\address{Institut f\"{u}r mathematische Logik und Grundlagenforschung\\Universit\"{a}t M\"{u}nster\\Einsteinstr.\ 62, 48149\\M\"{u}nster, Germany}
\email{rds@math.uni-muenster.de} 
\urladdr{https://ivv5hpp.uni-muenster.de/u/rds/}
\thanks{$^{\ddagger}$The author is funded by the Deutsche Forschungsgemeinschaft (DFG, German Research Foundation) under
	Germany's Excellence Strategy EXC 2044 390685587, Mathematics M\"unster: Dynamics - Geometry
	- Structure.}

\begin{document}

\maketitle

\begin{abstract}  Assuming that there is no inner model with a Woodin cardinal, we obtain a characterization of $\lambda$-tall cardinals in extender models that are iterable. In particular we prove that in such extender models, a cardinal $\kappa$ is a tall cardinal if and only if it is either a strong cardinal or a measurable
	limit of strong cardinals.
\end{abstract}


\section{Introduction}
\sectionmark{Equivalence}

Tall cardinals appeared in varying contexts as hypotheses in the work of Woodin and Gitik but they were only named as a distinct type of large cardinal by Hamkins in \cite{tall}. 

\begin{definition} Let $\alpha$ be an ordinal and $\kappa$ a cardinal. We say that $\kappa$ is \emph{$\alpha$-tall} iff there is an elementary embedding $j:V\rightarrow M$ such that the following holds:
	\begin{enumerate}
		\item[a)] $\crit(j)=\kappa$,
		\item[b)]  $j(\kappa) > \alpha$,
		\item[c)]  $^{\kappa}M \subseteq M$.
	\end{enumerate}
	We say that $\kappa$ is a \emph{tall cardinal} iff $\kappa$ is $\alpha$-tall for every ordinal $\alpha$.
\end{definition}

One can compare this notion with that of strong cardinals. 

\begin{definition}
	Let $\alpha$ be an ordinal and $\kappa$ a cardinal. We say that $\kappa$ is \emph{$\alpha$-strong}  iff there is an elementary embedding $j:V\rightarrow M$ such that the following holds:
	\begin{enumerate}
		\item[a)]  $\crit(j)=\kappa$,
		\item[b)]  $j(\kappa) > \alpha$,
		\item[c)]  $V_{\alpha} \subseteq M$.
	\end{enumerate}
	We say that $\kappa$ is a \emph{strong cardinal} iff $\kappa$ is $\alpha$-strong for every ordinal $\alpha$.
\end{definition}

In this paper, working under the hypothesis that there is no inner model with a Woodin cardinal, we present a characterization of $\lambda$-tall cardinals in `extender models' (see Definition \ref{DefExtMod}) that are `self-iterable' (see Definition  \ref{DefSelfIt}).

Given a cardinal $\kappa$, if $\kappa$ is $\alpha$-strong then $\kappa$ is $\alpha$-tall, and the existence of a strong cardinal is equiconsistent with the existence of a tall cardinal (see \cite{tall}). We will prove that the following equivalence holds in extender models:

\begin{cora}\label{cora} Suppose that there is no inner model with a Woodin cardinal, $V$ is an extender model of the form $L[E]$ which is iterable. Then given a cardinal $\kappa$ the following equivalence holds: $\kappa$ is a tall cardinal iff $\kappa$ is either a strong cardinal or  a measurable limit of strong cardinals.
\end{cora}

\begin{remark}
	In contrast to Corollary A,  Hamkins in  \cite[Theorem~4.1]{tall} adapted  Magidor's results in \cite{MR429566} to prove that it is consistent (assuming the consistency of $\zfc$ plus a strong cardinal) that there is a model of  $\zfc$ where there exists  a cardinal $\kappa$ such that $\kappa$ is a tall cardinal  which is neither a strong cardinal nor a limit of strong cardinals. Hence the equivalence from Corollary A does not hold in such model. 
\end{remark}

An extender\footnote{For an introduction to the theory of extender we recommend \cite{Kana}.} is a means of encoding elementary embeddings of models of (fragments of) $\zfc$ in a set-size object. There are various ways to represent extenders.  
Extenders are a generalization of measures, in particular, notions such as a `critical point' which are used in the context of measures can also be used when talking about extenders.

Extender models are a generalization of G\"odel's constructible universe that can accommodate large cardinals. In general, given a predicate $E$, which can be a set or a proper class, $L[E]$ is the smallest inner model\footnote{An inner model is a transitive proper class that models $\zf$.} closed under the operation $x \mapsto E \cap x$. Inner models of the form $L[E]$ can be stratified using the $J$-hierarchy:
\begin{itemize}
	\item $J_{\emptyset}^{E}=\emptyset$,
	\item $J_{\alpha+1}^{E}:= \rud_{E}(\{J_{\alpha}^{E}\}\cup J_{\alpha}^{E})$,
	\item $J_{\gamma}^{E}:= \bigcup_{\xi <\gamma}J_{\xi}^{E}$ for $\gamma$ a limit ordinal,
	\item $L[E] = \bigcup_{\xi \in \ord} J_{\xi}^{E}$
\end{itemize}

where $\rud_{E}$ is the closure under rudimentary functions\footnote{See \cite{MR2768688} for the definition of $\rud_{E}$.} and the function $x \mapsto E\cap x$.

We are interested in the special case where $E$ is such that  $E:\ord \rightarrow V$ and for every ordinal $\alpha$, either $E_{\alpha}=\emptyset$ or $E_{\alpha}$ is a partial extender (see Definition \ref{DefExtMod}). That is, $E$ is a `sequence of extenders'.

\begin{convention}
There are different ways of organizing sequences of extenders, we will use Jensen's $\lambda$-indexing (see Remark \ref{Indexing}).	
\end{convention}

\begin{definition}\label{DefO} Suppose $V$ is an extender model of the form $L[E]$. Given a cardinal $\kappa$ we define\footnote{Note that our definitions of $O(\kappa)$ and $o(\kappa)$  are not standard, because we do not only consider extenders which are total, but also consider partial extenders. For this reason, our definitions differ from those in other references such as \cite{MR1876087} and \cite{MR1423421}.} $  o(\kappa) := \text{otp}(\{ \beta \mid crit(E_{\beta}) = \kappa \})$ \index{$o(\kappa)$} and $O(\kappa) := sup\{ \beta \mid crit(E_{\beta}) = \kappa \}$.
\end{definition}

Our main result, Theorem A, is a level-by-level version of Corollary A. The statement of Theorem A uses the notion of $\mu$-stable premouse which is introduced in Definition \ref{stableDef}. 
 \begin{thma}\label{TallStrong} Suppose that there is no inner model with a Woodin cardinal and that the universe $V$ is an iterable extender model $L[E]$. Let $\kappa <  \mu$ be regular cardinals. Suppose further that $L[E]|\mu$ is $\mu$-stable above $\kappa$. \underline{Then} $  \kappa \ \text{is} \ \mu\text{-tall} $ iff \begin{gather*} o(\kappa) > \mu \\ \text{ or} \\ \Big( o(\kappa) > \kappa^{+} \wedge  sup\{ \nu < \kappa \mid  o(\nu) > \mu \} = \kappa  \Big) \end{gather*}
 \end{thma}
 
\vskip 1cm

We prove Theorem A in Section 4. The rest of this introduction gives a technical overview of our proof of Theorem A. 

In order to prove Theorem A we will need some results from core model theory. Specifically, we shall need the core model $\mathcal{K}$ below a Woodin cardinal. This model is an extender model\footnote{See \cite{HODasK} for an example where $\mathcal{K}$ is not an extender model.} that generalizes the covering, absoluteness, and definability properties of $L$. The following result due to Jensen and Steel guarantees that such a model exists. 

\begin{thm} (\cite{Knm}) \label{Knm} There are $\Sigma_2$ formulae $\psi_{\cK}(v)$ and $\psi_{\Sigma}(v)$ such that, if there is no inner model with a Woodin cardinal, then
		\begin{enumerate}
			\item[(i)] $\cK=\{v \mid \psi_{\cK}(v) \}$ is an inner model satisfying $\zfc$,
			\item[(ii)] $\psi_{\cK}^{V}=\psi_{\cK}^{V[g]}$, and $\psi_{\Sigma}^{V}=\psi_{\Sigma}^{V[g]}\cap V$, whenever $g$ is $V$-generic over a poset of set size,
		    \item[(iii)] For every singular strong limit cardinal $\kappa$,  $\kappa^{+}=(\kappa^+)^{\cK}$, 
		   	\item[(iv)] $\{v \mid \psi_{\Sigma}(v) \}$ is an iteration strategy for $\cK$ for set-sized iteration trees, and moreover the unique such strategy,
		   	\item[(v)] $\cK|\omega_{1}$ is $\Sigma_{1}$ definable over $J_{\omega_{1}}(\mathbb{R})$.
	\end{enumerate}
\end{thm}

We now describe our strategy for proving Theorem A.  One direction is due to Hamkins (see Theorem~\ref{Hamkins} and Theorem~\ref{PropI}), so we start from the assumption that $\kappa$ and $\mu$ are cardinals such that $\mu > \kappa$, $\kappa$ is $\mu$-tall and $j$ witnesses that $\kappa$ is $\mu$-tall. Note that this implies that $\kappa$ is measurable and that $\mu <j(\kappa)$. 

We now consider two cases. Either $\kappa$ is a limit of cardinals $\beta$ such that $o(\beta) > \mu$, in which case we get the second alternative of the direction we are proving using Lemma~\ref{TallMeas}. 

So, suppose that $\kappa$ is not a limit of cardinals $\beta$ such that $o(\beta) > \mu$, and then we work towards proving that $o(\kappa) > \mu$. 

As a first step, we combine Lemma \ref{lemma} and Theorem \ref{VM} to obtain that $j$ is an iteration map coming from an iteration tree $\Tt $ on $L[E]$ such that $j=\pi^{\Tt}_{0,\infty}$. This is Lemma~\ref{InducedIteration}. We will spend the main part of the proof of Theorem A analyzing the iteration tree $\Tt$.  

We shall prove that $o(\kappa) > \mu$ by contradiction. That is, we shall start with the assumption that $o(\kappa)\leq \mu$. Then, we will find $\Theta$ and $\beta^*$ such that $\beta^* < \kappa \leq \Theta \leq \mu$ and $\tau \in (\beta^{*},\Theta] $ implies $o(\tau) < \Theta$ (see Lemma \ref{gap}).  Using the upper bounds that we obtain in Section \ref{Section3} we shall finally prove that $j(\kappa)=\pi^{\Tt}_{0,\infty}(\kappa) \leq \Theta \leq \mu$, which will contradict the fact that $\mu < j(\kappa) =\pi^{\Tt}_{0,\infty}(\kappa)$.

The following is essentially a reformulation of Theorem A which we also prove in Section 4. 
\begin{thmb}\label{thmb} Suppose there is no inner model with a Woodin cardinal and $L[E]$ is an extender model that is self-iterable. Let $\kappa$, $\mu$ be ordinals such that $\kappa < \mu$ and $\mu$ is a regular cardinal. 
	If $L[E]|\mu$ is $\mu$-stable above $\kappa$, then $(\kappa \text{ is } \mu\text{-tall})^{L[E]}$ iff 
	\begin{equation}
	\begin{gathered} \label{EquivalenceEquation} (o(\kappa)) > \mu)^{L[E]} \\ \text{or} \\ (o(\kappa) > 0 \wedge \kappa = \sup\{\nu < \kappa \mid o(\nu)>\mu\})^{L[E]}
	\end{gathered}
	\end{equation}
	It follows that if $L[E]$ is weakly iterable and $L[E]|\mu$ is $\mu$-stable above $\kappa$, then $(\kappa \text{ is } \mu\text{-tall})^{L[E]}$ iff \eqref{EquivalenceEquation} holds.
\end{thmb}

\section{Preliminaries} 

In this section we summarize the notation that will be used in this paper. We follow closely the notation used in \cite{MR1876087}.

\begin{definition}
	We say that $\mathcal{M}:=\langle J_{\alpha}^{E}, \in, E\restriction \alpha, E_{\alpha} \rangle$ is an \emph{acceptable $J$-structure} iff $\mathcal{M}$ is a transitive amenable structure and for every $\xi < \alpha$ and $\tau < \alpha \omega$ if $(\mathcal{P}(\tau) \cap J_{\xi+1}^{E}) \setminus J_{\xi}^{E} \neq \emptyset$, then there is $f:\tau \rightarrow J_{\xi}^{E}$ surjective such that $f \in J_{\xi+1}^{E}$. 
\end{definition} 

Acceptablity is a strong form of $\gch$. We can define fine structure for acceptable $J$-structures. Let $\cM$ be an acceptable $J$-structure. We shall write (see \cite[Chapter 2]{MR1876087} ):

\begin{itemize}
	\item $\h(\cM)$ for the ordinal $\cM \cap \ord$,
	\item $\rho_{n}(\cM)$ for the $n$-th projectum of $\cM$,
	\item $P^{\cM}_{n}$ for the set of good parameters (i.e., for the set of parameters witnessing	$\rho_{n}(\cM)$ is the $n$-th projectum),
	\item $p^{\cM}_{n}$ for the $n$-th standard parameter of $\cM$ (i.e, the least element of $P^{\cM}_{n}$ where least refers to the canonical well-order of $[\ord]^{<\omega}$),
	\item $h^{n,p}_{\cM}$ for the canonical $\Sigma_{1}$ Skolem function of $\cM^{n,p}$,
	\item $\tilde{h}^{n}_{\cM}$ for the good uniformly $\Sigma^{(n-1)}_{1}(\cM)$ function with two parameters which is the result of iterated composition of the Skolem functions of the $i$-th reducts.	
\end{itemize}

\begin{definition}
	Let $\cM = \langle J_{\alpha}^{A},\in,F \rangle $ be an acceptable $J$-structure. We say that $\cM$ is a coherent $J$-structure iff there is an $\bar{\alpha} < \alpha$ \begin{itemize}
		\item $F$ is a whole extender\footnote{See \cite[p.42]{MR1876087} for the definition of extender and \cite[p. 53]{MR1876087} for the definition of whole extender.} in $J_{\bar{\alpha}}^{A}$ where $\bar{\alpha}< \alpha$,
			 \item $J_{\bar{\alpha}}^{A} \models ``\crit(F)$ is the largest cardinal''
			\item $J_{\alpha}^{A}=\Ult_{0}(J_{\bar{\alpha}}^{A},F)$.
	\end{itemize}
     
     Given $\beta < \alpha $ we define $\cM|\beta:= \langle J_{\beta}^{A}, \in, E\restriction \omega \beta \rangle$ and $\cM||\beta:= \langle J_{\beta}^{A}, \in, E\restriction \omega \beta, E_{\omega\beta} \rangle$\footnote{Depending on the reference $\cM||\beta$ and $\cM|\beta$ may have their roles switched. We stick to the notation in \cite{MR1876087}.}.
    
	We say that $\mathcal{M}:= \langle J_{\alpha}^{E},\in,E\restriction \omega \alpha, E_{\omega \alpha} \rangle $ is a premouse iff
	\begin{itemize}
		\item $E$ is a set of triples $\langle \nu,x,y \rangle$ for $\nu \leq \alpha$ such that, setting 
		$$E_{\omega \nu}:=\{\langle x,y \rangle \mid \langle \nu, x,y \rangle \in E \},$$ the structure $\mathcal{M}||\nu$ is coherent whenever $E_{\omega \nu} \neq \emptyset$.
		\item For every $\nu \leq \alpha$, if $E_{\omega \nu} \neq \emptyset$, then $E_{\omega \nu}$ is weakly amenable w.r.t. $\mathcal{M}||\nu$.
		\item $\mathcal{M}||\nu$ is sound for every $\nu < \alpha$.
	
	\end{itemize}
\end{definition}

\begin{remark}[Indexing] \label{Indexing} Notice that implicitly in our definition of a premouse we use $\lambda$-indexing, also called Jensen indexing, which means that  extenders are indexed  at the successor of the image of their critical point under the ultrapower map, i.e., if $\mathcal{M}$ is a premouse, $ E^{\mathcal{M}}_{\beta}\neq \emptyset$, $\mathcal{N} = Ult_{0}(\mathcal{M}||\beta,E_{\beta}^{\mathcal{M}})$ and $\pi_{E_{\beta}^{\cM}}:\cM||\beta \rightarrow \cN$ is the ultrapower map, then $\beta = \pi_{E_{\beta}^{\mathcal{M}}}(\crit(E_{\beta}^{\mathcal{M}}))^{+\mathcal{N}}$.
\end{remark}

\begin{definition}\label{lambda}
	Let $F$ be an extender over a premouse $\cM$. We denote by $\lambda(F)$ the image of the critical point of $F$ under the ultrapower map, i.e. if $\pi_{F}: \cM \rightarrow \Ult_{0}(\cM,F)$ is the ultrapower map, we let $\lambda(F) = \pi_{F}(\crit(F))$. 
\end{definition}

\begin{definition}\label{DefExtMod} We say that $L[E]$ is an \emph{extender model} iff $L[E]$ is a proper class premouse. 	
\end{definition}

\section{Upper bounds for the images of ordinals under iteration maps \label{Section3} }

In this section we define iteration trees\footnote{The reader interested in the intuition behind the definition of iteration trees is refered to  \cite{MR1224594}.} and prove general facts about upper bounds for the images of ordinals under iteration maps.

\begin{definition}
	A tree $T=\langle \theta, \leq_{T} \rangle $ on an ordinal $\theta$ is an \emph{iteration tree} iff 
	\begin{itemize}
		\item[a)] $0$ is the root of $T$ and each successor ordinal $\alpha < \theta$ has an immediate $T$-predecessor $\pred_{T}(\alpha) < \alpha$;
		\item[b)] if $\alpha < \theta$ is a limit ordinal, then $\alpha = sup\{\xi < \alpha \mid \xi <_{T} \alpha \}$
	\end{itemize}
If $T \subseteq \theta$ is an iteration tree, given $\alpha < \beta$ elements of $T$ we write 
$(\alpha,\beta]_{T}=\{\gamma \mid \alpha <_{T} \gamma \leq_{T} \beta \}$ and similarly for $(\alpha,\beta)_{T}$, $[\alpha,\beta]_{T}$.
\end{definition}

\begin{definition} \label{ItTrees} Let $\cM$ be a sound premouse and $\theta \in \ord$. We say that $\Tt$ is an \emph{iteration tree $\Tt$ on $\cM$ with $\lh(\Tt)=\theta$} iff $\Tt$ is a 6-tuple\footnote{When $\cM$ is a proper class premouse we allow $\eta_{\alpha}=\ord$, and formally we use $\eta_{\alpha}=\emptyset$.}:
	
	$$\mathcal{T}= \langle \langle \cM_{\alpha} \mid \alpha < \theta \rangle, \langle \nu_{\beta} \mid \beta+1 \in \theta \rangle, \langle \eta_{\beta} \mid \beta+1 < \theta \rangle, \langle \pi_{\alpha,\beta} \mid \alpha \leq_{T} \beta < \theta \rangle, D, T \rangle $$	

satisfying: 

\begin{enumerate}
	\item[(a)] $T$ is an iteration tree.
	\item[(b)] Each $\cM_{\alpha}$ is a premouse and $\cM_{0}=\cM$.
	\item[(c)] $\langle \pi_{\alpha,\beta} \mid \alpha \leq_{T} \beta \rangle $ is a commutative system of partial maps, where $\pi_{\alpha,\beta}:\cM_{\alpha} \rightarrow \cM_{\beta}$. 
	\item[(d)] Setting $\xi_{\alpha}=\pred_{T}(\alpha+1)$, we have $\eta_{\alpha} \leq ht(\cM_{\xi_{\alpha}})$ and for every $\beta < \theta$ there are only finitely many $\alpha$ such that  $\xi_{\alpha} <_{T} \beta $ and $\eta_{\alpha} < ht(\cM_{\xi_{\alpha}})$. 
	\item[(e)] $D\subseteq \theta$ and  $\alpha \in D $ iff $  \eta_{\alpha} < \h(\cM_{\xi_{\alpha}}) $.
    \item[(f)] If $\alpha+1 < \theta$, setting $\kappa_{\alpha}:= \crit(E_{\nu_{\alpha}}^{\cM_{\alpha}})$ and $\tau_{\alpha}:=\kappa_{\alpha}^{+\cM_{\alpha}||\nu_{\alpha}}$, we have $\tau_{\alpha}=\kappa_{\alpha}^{+\cM_{\xi_{\alpha}}||\eta_{\alpha}}$, $E^{\cM_{\alpha}}\restriction \tau_{\alpha}=E^{\cM_{\xi_{\alpha}}}\restriction \tau_{\alpha}$, and 
	$$\pi_{\xi_{\alpha},\alpha+1}: \cM_{\xi_{\alpha}}|| \eta_{\alpha} \longrightarrow_{E_{\alpha}}^{*} \Ult^{*}(\cM_{\xi_{\alpha}}||\eta_{\alpha} , E_{\nu_{\alpha}}^{\cM_{\alpha}} ).$$
	\item[(g)] If $\alpha < \theta$ is a limit ordinal, then $\langle M_{\alpha}, \pi_{\beta,\alpha} \mid \beta < \alpha \rangle $ is the direct limit of the diagram $\langle \cM_{\beta} , \pi_{\bar{\beta},\beta} \mid \bar{\beta} \leq_{T} \beta <_{T} \alpha \rangle $.
\end{enumerate}

Given an iteration $\Tt$, we denote the objects from the above definition related to $\Tt$ by $\cM_{\alpha}^{\Tt}$, $\nu_{\alpha}^{\Tt}$, $\eta_{\alpha}^{\Tt}$,$D^{\Tt}$,$T^{\Tt}$. We shall often write $T$ instead of $T^{\Tt}$. We also set:
\begin{itemize}
	\item $E^{\Tt}_{\alpha} := E^{\cM^{\Tt}_{\alpha}}_{\nu_{\alpha}}$ (the extender used to form $\cM_{\alpha+1}^{\Tt})$,
	\item $\kappa^{\Tt}_{\alpha} := \crit(E_{\alpha}^{\Tt})$,
	\item $\lambda_{\alpha}^{\Tt} := \lambda(E_{\alpha}^{\Tt})$,
	\item  $\tau_{\alpha}^{\Tt} := ((\kappa_{\alpha}^{\Tt})^{+})^{\cM_{\alpha}^{\Tt}||\nu_{\alpha}^{\Tt}}$ and 
	\item $D^{\Tt} := \{\alpha+1 \mid \eta_{\alpha}^{\Tt}< ht(\cM_{\xi_{\alpha}}^{\Tt})\}$.
\end{itemize}

If in addition $\lh(\Tt)=\gamma+1$ for some $\gamma \in \ord$, we write
\begin{itemize}
	\item  $\cM_{\infty}^{\Tt}$ for the last model  $\cM_{\gamma}^{\Tt}$ in the iteration tree $\Tt$,
	\item $[0,\gamma]_{T}$ is called the main branch of $\Tt$, and 
	\item if $\beta \in \lh(\Tt)$ and $\beta <_{T} \gamma$, we let $\pi_{\beta,\infty}^{\Tt}:=\pi_{\beta,\gamma}^{\Tt}$.
\end{itemize}

We say that $\Tt$ is a \emph{normal iteration tree} iff 

\begin{itemize}
	\item $\nu_{\beta} <\nu_{\alpha}$ whenever $\beta, \alpha \in \lh(\Tt)$ and $\beta < \alpha$;
	\item $\xi_{\alpha}=$ the least $\xi \in B$ such that $\kappa_{\alpha} < \lambda_{\xi}$;
	\item $\eta_{\alpha}$= the maximal $\eta \leq ht(M_{\xi_{\alpha}})$ such that $\tau_{\alpha}=\kappa_{\alpha}^{+\cM_{\xi_{\alpha}}^{\Tt}||\eta}$. 
\end{itemize}
\end{definition}

\begin{remark} We stress that the maps $\pi_{\bar{\beta},\beta}^{\Tt}$ are partial functions.
	\end{remark} 

\begin{fact}\label{fact1} If $\Tt$ is a normal iteration tree on a premouse $\cM$, then the inductive application of the coherency condition yields: 
	\begin{itemize}
		\item $\cM_{\alpha}^{\Tt}| \nu_{\beta}=\cM^{\Tt}_{\beta} |\nu_{\beta}$ whenever  $\beta \leq \alpha$.
		\item If  $\beta < \alpha$, then $\nu_{\beta}$ is a successor cardinal in $\cM_{\alpha}^{\Tt}$, but  not a cardinal in $\cM_{\beta}^{\Tt}$ when  $\nu_{\beta}^{\Tt} \in \cM_{\beta}^{\Tt}$.
	\end{itemize}
	\end{fact}
\begin{convention} In this paper all iteration trees that we will encounter are normal iteration trees. In what follows when we write \textit{iteration tree} we mean \textit{normal iteration tree}.
\end{convention}

\begin{definition}
Let $\cM$ be a premouse and $\Tt $ an iteration tree on $\cM$ such that $\lh(\Tt)$ is a limit ordinal. We say that $b$ is a \emph{cofinal wellfounded branch through $\Tt$} iff 
 \begin{itemize}
		\item $b$ is a branch through $T^{\Tt}$ cofinal in $\lh(\Tt)$,
		\item $D^{\Tt}\cap b $ is finite,
		\item the direct limit along $b$ is well-founded.
	\end{itemize}

\end{definition}

\begin{remark}
	As we work under the hypothesis that there is no inner model with a Woodin cardinal, it follows that if $\Tt$ is  an iteration tree on a premouse $\cM$ and $\lh(\Tt)$ is a limit ordinal then $\Tt$ has at most one cofinal wellfounded branch through $\Tt$ (see  \cite{MR1876087}[Corollary 9.4.7], \cite{MR2768698}[Theorem 6.10]). In order to simplify notation  we will avoid mentioning iteration strategies in our definition of iterability but we warn the reader that this is not how iterability is usually defined.  In general, we would need a strategy to choose cofinal branches at limit stages, see \cite{MR1876087}[p. 290]. For the reader familiar with iteration strategies, as we work under the hypothesis that there is no inner model with a Woodin cardinal, the strategy of any premouse we encounter will always choose the unique cofinal wellfounded branch.
\end{remark}

\begin{definition}[Normal iterability]
Let $\alpha$ be a limit ordinal or the proper class $\ord$. 	
We say that $\cM$ is normaly \emph{$\alpha$-iterable} iff for any normal iteration tree $\Tt$  on $\cM$ of length $< \alpha$ the following holds:
\begin{enumerate}
	\item[(a)] If $\Tt$ is of limit length, then there is  a cofinal wellfounded branch $b$ through $\Tt$.
	\item[(b)] If $\Tt$ is of successor length $\gamma$ and $\gamma + 1 < \alpha$ and $\nu \geq \sup\{\nu_{\alpha}^{\Tt} \mid \alpha <\gamma\}$ is such that $E_{\nu}^{\cM^{\Tt}_{\gamma-1}} \neq \emptyset$, then $\Tt$ has a normal extension $\Tt'$ of length $\gamma+1$ with $\nu_{\gamma-1}^{\Tt'}:=\nu$. In other words, setting $\nu_{\gamma-1}^{\Tt}=\nu$, the ultrapower $\Ult^{*}(\cM_{\xi_{\gamma-1}^{\Tt'}}^{\Tt'}||\eta_{\gamma-1}^{\Tt'}, E_{\nu^{\Tt'}_{\gamma-1}}^{\cM^{\Tt'}_{\gamma}})$ is well founded, where $\eta^{\Tt'}_{\gamma-1}$ and $\xi^{\Tt'}_{\gamma-1}$ are determined by the rules of normal iteration trees.
\end{enumerate}
	When $\alpha=\ord$ and $\cM$ is normaly $\ord$-iterable we shall omit $\ord$ and write that $\cM$ is \emph{normaly iterable}.
\end{definition}

We shall also need a stronger notion of iterability .  

\begin{definition}
	Let $\cM$ be a premouse and $n \in \omega $. We say that $\vec{T}=\langle  \Tt_{k} \mid k \leq  n \rangle $ is a \emph{stack of normal iteration trees on $\cM$} iff $\Tt_{0}$ is an iteration tree on $\cM_{0}^{0}=\cM$ and  for every $ k < n, \cM^{k}_{0} \triangleleft \cM_{\infty}^{\Tt_{k-1}}$ and $\Tt_{k}$ is a normal iteration tree on $\cM_{0}^{k}$.
	\end{definition}
\begin{definition}[Iterability for stacks of normal trees]
	Let $\cM$ be a premouse. We say that $\cM$ is \emph{iterable} iff 
	\begin{enumerate}
		\item[(a)] If $\Tt= \langle \Tt_{k} \mid k\leq n \rangle $ is a stack of iteration trees on $\cM$, then $\cM_{\infty}^{\Tt_{n}}$ is normaly iterable. 
		\item[(b)] If $\Tt= \langle \Tt_{n} \mid n \in \omega \rangle $ is a stack of iteration trees on $\cM$ such that for each $n \in \omega$, $\cM_{\infty}^{\Tt_{n}}$ is normaly iterable, then for all sufficiently large $n \in \omega$ we have that $D^{\Tt_{n}} \cap b^{\Tt_{n}}=\emptyset$, where $b^{\Tt_{n}}$ is the main branch of $\Tt$ so that $\tau_{n}: \cM^{n}_{0}\rightarrow \cM^{\Tt_{n}}_{\infty}=\cM^{n+1}_{0}$ is defined for all $n$ sufficiently large. Moreover, the direct limit of the $\cM_{0}^{n}$'s under the $\tau_{n}$'s is wellfounded.
	\end{enumerate} 
	
\end{definition}

\begin{definition}[Self-iterability] \label{DefSelfIt} If $L[E]$ is an extender model, we say that $L[E]$ is \emph{self-iterable} iff the following sentence in the language $\{\in,\dot{E}\}$ holds: $$(\forall \alpha (\alpha \in \ord \rightarrow \langle J_{\alpha}^{E},\in,\dot{E}\restriction \omega\alpha, \dot{E}_{\omega\alpha} \rangle \text{ is iterable} ))^{L[E]}$$	
\end{definition}
The following definition is a slight variation of the notion of \textit{stable premouse} defined in \cite{Knm}.

\begin{definition} \label{stableDef}
	Let $\cM$ be a premouse and $\mu$ a regular cardinal and $\kappa \in \mu$. If $\cM \cap \ord \leq \mu$ we say that $\cM$ is \emph{$\mu$-stable} above $\kappa$ iff one of the following holds:
	\begin{enumerate}
		\item[(1)] $\cM\cap \ord < \mu$, or
		\item[(2)] $\cM\cap \ord = \mu$ and one of the following holds:
		\begin{enumerate}
			\item[(a)] $(\text{There is no largest cardinal})^{\cM}$, or
			\item[(b)] There is $\gamma < \mu$ such that $(\gamma \text{ is the largest cardinal } \wedge \cf(\gamma) < \kappa)^{\cM}$, or 
			\item[(c)] There is $\gamma < \mu$ such that $(\gamma \text{ is the largest cardinal } \wedge \cf(\gamma) \geq \kappa)^{\cM}$ and there is no $\beta$ such that $E_{\beta}^{\cM}$ is a total measure on $\cM$ with critical point $\cf^{\cM}(\gamma)$.
		\end{enumerate}
	\end{enumerate}
\end{definition}

 The next lemma shows that given a regular cardinal $\mu$, under very general conditions the ultrapower of a premouse $\cM$ of height $\leq \mu$ by an extender $F$ with $\lambda(F) < \mu$, has height $\leq \mu$.
\begin{lemma} \label{Stability} Let $\mu$ be a regular cardinal and $\cM$ a sound\footnote{Soundness implies that $\cM= h^{n}_{\cM}(\rho_{n}(\cM)\cup\{p^{n,\cM}\})$ for all $n \in \omega$.} premouse such that $\cM \cap \ord \leq \mu$. Let   $\kappa < \alpha \leq \mu$ and $F$ be  such that: 
	\begin{itemize}
		\item  $F$ is an extender over $\cM||\alpha$ and $\alpha$ is the largest ordinal $\leq \mu$ with this property.
		\item $Ult_{n}(\cM||\alpha,F)$ is well founded, where $n$ is the largest $k\leq\omega$ such that $\crit(F) <\rho_{k}(\cM)$.
		\item $\lambda(F) < \mu$.
		\item $ (\kappa \leq \crit(F) \wedge \crit(F)^{+} \text{ exists} )^{\cM||\alpha}$.
	\end{itemize} 
	 Suppose further that if $\alpha = \mu$ and there is $\gamma < \alpha$ such that  $$(\gamma \text{ is the largest cardinal})^{\cM||\alpha}$$ and $(cf(\gamma) \geq \kappa)^{\cM||\alpha}$,
	then $(crit(F) \neq cf(\gamma))^{\cM||\alpha} $. 
	
Then $$  Ult_{n}(\cM||\alpha,F) \cap \ord \leq \mu. $$  Moreover if $\alpha < \mu$ the above inequality is strict and if $\alpha = \mu$ then equality holds. 
\end{lemma}
\begin{proof}
	We split the analysis into two cases: $\alpha < \mu$ and $\alpha = \mu$ and the second case splits further into two subcases.  

	$\bullet$ Suppose $\alpha < \mu$. Notice that for $n > 0 $ the set $$\{ f:\crit(F) \rightarrow \cM||\alpha \mid  f \in \Sigma_{1}^{(n-1)}(\cM||\alpha) \}$$ has cardinality $\leq |\alpha|$.
	Therefore 
	$$|\Ult_{n}(\cM||\alpha,F)| \leq \max\{|\alpha|, |\lambda(F)| \} < \mu,$$
	 which implies $\Ult_{n}(\cM||\alpha,F) \cap \ord < \mu$. 
	
	$\bullet$ Suppose $\alpha = \mu$. Since $\cM $ is sound,  $\rho_{\omega}(\cM)=\mu$ and $n=\omega$. Let $i_{F}:\cM \rightarrow \Ult_{0}(\cM,F)$ be the ultrapower map derived from $F$. 
	We split this case into two subcases: 
	
	$\blacktriangleright$ Suppose there is $\gamma$ such that $(\gamma \text{ is the largest cardinal})^{\cM}$:	
	 \begin{claim}
	 	If $i_{F}(\gamma) < \mu$, then $\Ult_{0}(\cM,F) \cap \ord \leq \mu$
	 \end{claim} 
	 \begin{proof}
	 	For a contradiction suppose that $i_{F}(\gamma) < \mu$ and that there is $\xi \in \Ult_{0}(\cM,F) $ such that $ \xi \geq \mu $. The ultrapower map $i_{F}$ is cofinal in $\Ult_{0}(\cM,F)$, therefore there is $\beta \in \mu$ such that $i_{F}(\beta) \geq \xi \geq \mu$. Let $$\varphi(\beta,\gamma,h):=  ``(h:\gamma \rightarrow \beta) ~ \wedge ~ ( h \text{ is a surjection})"$$ The formula $\exists h \varphi (\beta,\gamma,h)$ is $\Sigma_{1}$ and therefore it is preserved by $i_{F}$ and
	 	\begin{gather}\label{otimes} \Ult_{0}(\cM,F) \models \exists h \varphi(i_{F}(\beta),i_{F}(\gamma),h). \end{gather} 
	 	Fix $h \in Ult_{0}(\cM,F)$ which witnesses \eqref{otimes}. As  $\varphi(i_{F}(\beta,i_{F}(\gamma)),h)$ is $\Sigma_{0}$ and $\Ult_{0}(\cM,F)$ is transitive it follows that $ \varphi(i_{F}(\beta,i_{F}(\gamma)),h)$ holds in $V$. Therefore $h$ is a surjection from $i_{F}(\gamma)$ onto $i_{F}(\beta)$. As $i_{F}(\gamma) < \mu$ and $i_{F}(\beta) > \mu$ this contradicts our hypothesis that $\mu$ is a cardinal. 
	 \end{proof}
	 
	 Next we verify $i_{F}(\gamma)= sup_{\xi < \gamma}(\xi)$. 
	 Let $\zeta \in i_{F}(\gamma)$ and let $f \in \cM$, $a \in \lambda(F)^{<\omega}$ be such that $f:\crit(F) \rightarrow \gamma $ and $ [a,f] $ represents $\zeta$ in the ultrapower of $\cM$ by $F$. From the hypothesis in our lemma, if $cf^{\cM}(\gamma) \leq \kappa $ or $cf^{\cM}(\gamma) > \kappa$ in both cases we have $cf^{\cM}(\gamma) \neq \crit(F)$.
	 Then (i) or (ii) below must hold:
	  	  \begin{enumerate}
	  	\item[(i)]  $cf^{\cM}(\gamma)>\crit(F)$ implies that there is $\xi < \gamma $ such that $sup(ran(f))<\xi$,
	  	\item[(ii)] $cf^{\cM}(\gamma) <\crit(F)$ implies that there is $\xi<\gamma$ such that $\{u\in \crit(F)^{|a|} \mid f(u) \in \xi \} \in F_{a}$. 
	  	\end{enumerate} 
	    Thus we can find $\xi < \gamma $ such that $[a,f] \in i_{F}(\xi) \in i_{F}(\gamma)$.    Hence $i_{F}(\gamma) = sup_{\xi<\gamma}i_{F}(\xi)$. 
	   
	   \begin{claim}\label{iFgamma} Given $\xi < \gamma$, it follows that $$|i_{F}(\xi)| \leq \max\{(|\xi^{\crit(F)}|)^{\cM},\crit(F)^{+\cM},|\lambda(F)|^{\cM}\} \leq \gamma <\mu.$$
	   \end{claim}
	   \begin{proof} From our hypothesis that $\lambda(F) <\mu$, since $\gamma$ is the largest cardinal of $\cM$, it follows that $|\lambda(F)|^{\cM} \leq \gamma$. 
	   	
	   	From our hypothesis that $\crit(F)^{+\cM}$ exists in $\cM$, it follows that $\crit(F)^{+} \leq \gamma$.
	   	
	   	Notice that $\xi < \gamma$ and $\delta < \gamma$ imply  $|\xi^{\crit(F)}|^{\cM} \leq \gamma$. 
	   	\end{proof} 
	   	
	   		   	Therefore from the above claim and the regularity of $\mu$ we have:	   $$i_{F}(\gamma) = sup_{\xi<\gamma}i_{F}(\xi) <\mu.$$

	 $\blacktriangleright$ Suppose $(\text{there is no largest cardinal})^{\cM}$. Since $i_{F}$ is cofinal in $\Ult_{0}(\cM,F)$ it will be enough to verify that $i_{F}(\xi) <\mu$ for all $\xi < \mu$.   Given $\xi < \mu $, similarly as in the proof of Claim \ref{iFgamma} we get that $ |i_{F}(\xi)| \leq \max\{(|\xi^{\crit(F)}|)^{\cM},\crit(F)^{+\cM},|\lambda(F)|\} < \mu$. 
	\end{proof}
Using induction and Lemma \ref{Stability} we can obtain the following:

\begin{lemma}(\cite[Lemma 4.8]{Knm}) \label{StablePM}
	Let $\mu$ be a regular cardinal in $V$, $\kappa$ an ordinal such that $\kappa < \mu$ and  $\cM$ is a sound premouse that is $\mu$-stable above $\kappa$. Let $\Tt$ be an iteration tree on $\cM$ such that $\lh(\Tt)<\mu$ and, for all $\beta+1 < \lh(\Tt) $, $\crit(E_{\beta}^{\Tt}) \geq \kappa$.  	
	Then  $\beta \in lh(\Tt)$ implies $\cM_{\beta}^{\Tt}\cap \ord \leq \mu$.
\end{lemma}

\begin{definition}\label{DefLives}
	Given a premouse $\cM$,  $\mu \in \cM \cap \ord$ and a normal iteration tree $$\Tt= \langle \langle \cM_{\alpha}^{\Tt} \mid \alpha < \theta \rangle, \langle \nu_{\beta}^{\Tt} \mid \beta+1 < \theta \rangle, \langle \eta^{\Tt}_{\beta} \mid \beta+1 < \theta \rangle, \langle \pi^{\Tt}_{\alpha,\beta} \mid \alpha \leq_{T} \beta < \theta \rangle, T^{\Tt} \rangle $$ on $\cM$, we say that \emph{$\Tt$ lives on $\cM|\mu$} iff
	$$\mathcal{U}:= \langle \langle \cM_{\alpha}^{\mathcal{U}} \mid \alpha < \theta \rangle, \langle \nu^{\mathcal{U}}_{\beta} \mid \beta+1 <  \theta \rangle, \langle \eta_{\beta}^{\mathcal{U}} \mid \beta+1 < \theta \rangle, \langle \pi^{\mathcal{U}}_{\alpha,\beta} \mid \alpha \leq_{T^{\mathcal{U}}} \beta < \theta \rangle, T^{\mathcal{U}} \rangle $$	 
	such that $\cM_{0}^{\mathcal{U}}=\cM|\mu$,  $T^{\Tt}=T^{\mathcal{U}}$ and $\langle \nu_{\beta}^{\mathcal{U}} \mid \beta < \theta \rangle = \langle \nu_{\beta}^{\Tt} \mid \beta < \theta \rangle$ is a normal iteration tree on $\cM|\mu$. 
	
	We will denote $\mathcal{U}$ by $\Tt\restriction (\cM|\mu)$ and call it the \emph{restriction of $\Tt$ to $\cM|\mu$}.
\end{definition}

\begin{remark}
	In Definition \ref{DefLives} the sequence $\langle \eta^{\mathcal{U}}_{\beta} \mid \beta \in B^{\mathcal{U}} \rangle $ is determined by the other parameters in $\mathcal{U}$ and the requirement that $\mathcal{U}$ is normal. 
\end{remark}

\begin{lemma} \label{IterationBounds} Let $\mu$ be a regular cardinal, $\kappa$ an ordinal such that $\kappa < \mu$ and  $\cM$ be a proper class premouse. Suppose that $\cM|\mu$ is $\mu$-stable above $\kappa$.  Let  $\Tt$ be an iteration tree on $\cM$ that lives on $\cM|\mu$ and for all $\beta +1 \in \lh(\Tt)$ we have $\crit(E_{\beta}^{\Tt}) \geq \kappa$ and $\lh(\Tt)<\mu$. Let $\Uu= \Tt \restriction \cM|\mu$. Then for all $\beta \in \lh(\Tt)$, the following holds:
	\begin{itemize}
		\item[$(a)_{\beta}$] If $D^{\Tt}\cap[0,\beta]_{T}=\emptyset$, then 
		\begin{itemize}
			\item $\pi_{0,\beta}^{\Tt}(\mu)$ is defined,
			\item  $\pi_{0,\beta}^{\Tt}(\mu)= \sup_{\gamma < \mu} (\pi^{\Tt}_{0,\beta}(\gamma)) = \mu$
			\item  $\cM_{\beta}^{\Tt}|\mu = \cM_{\beta}^{\Uu}$.
		\end{itemize}
	\item[$(b)_{\beta}$] If $D^{\Tt} \cap [0,\beta]_{T}\neq \emptyset$, then $\cM_{\beta}^{\Tt}=\cM^{\Uu}_{\beta}$.
	\end{itemize}
	
\end{lemma}

\begin{proof}
	We proceed by induction. Suppose $\beta \in \lh(\Tt)$ and $(a)_{\gamma} $ and $(b)_{\gamma}$ holds for all $\gamma < \beta$.  We split the analysis into two cases, $\beta$ is a successor ordinal or $\beta$ is a limit ordinal. 
	
	$\bullet$ Suppose $\beta = \delta+1 $ for some ordinal $\delta$. We again need to divide into two subcases based on whether $D^{\Tt}\cap [0,\beta]_{\Tt}$ is empty or not.
	
	$\blacktriangleright$ Suppose $D^{\Tt} \cap [0,\beta]_{T}=\emptyset$. Then $\pi_{0,\beta}^{\Tt}(\mu)$ is defined. Recalling our notation from Definition \ref{ItTrees}, $\xi_{\delta}^{\Tt}=\pred_{T}(\beta)$ and $\eta_{\delta}^{\Tt}$ is the largest ordinal such that $E_{\delta}^{\Tt}$ is a total extender over  $\cM_{\xi_{\beta}^{\Tt}}^{\Tt}||\eta_{\delta}^{\Tt}$. 
	
	As $D^{\Tt} \cap [0,\beta]_{T} =\emptyset$, it follows that $\cM_{\xi_{\delta}^{\Tt}}^{\Tt}$ is a proper class and $\cM_{\beta}^{\Tt}$ is the $\Sigma_0$-ultrapower of $\cM_{\xi_{\delta}^{\Tt}}^{\Tt}$ by $E_{\delta}^{\Tt}$. 
		
    Hence, given $\zeta < \pi_{\xi_{\delta},\delta}^{\Tt}(\mu) $, there are $a \in lh(E_{\delta}^{\Tt})^{<\omega}$ and $f \in \cM_{\xi_{\delta}^{\Tt}}^{\Tt}$, $f:(\crit(E_{\delta}^{\Tt}))^{|a|}\rightarrow \mu$ such that $[a,f]_{E_{\delta}^{\Tt}}$ represents $\zeta$ in $\Ult_{0}(\cM_{\xi_{\delta}^{\Tt}}^{\Tt},E_{\delta}^{\Tt})$. 
    
    As $\mu$ is a regular cardinal there is $\Upsilon < \mu$ such that $\sup(\ran(f)) < \Upsilon < \mu$ and therefore $\zeta < \pi_{\xi_{\delta},\beta}^{\Tt}(\Upsilon) < \pi_{\xi_{\delta},\beta}^{\Tt}(\mu)$. 
    
    By our induction hypothesis we have $$ \cM_{\xi_{\delta}^{\Tt}}^{\Tt}|\mu = \cM_{\xi_{\delta}}^{\Uu} \text{ and } E_{\delta}^{\Tt} = E_{\delta}^{\Uu}.$$
    
    Therefore, $$\cM_{\beta}^{\Uu}= \Ult_{0}(\cM_{\xi_{\delta}}^{\Uu},E_{\delta}^{\Tt})= \Ult^{*}(\cM_{\xi_{\delta}}^{\Tt}|\mu,E_{\delta}^{\Tt}),$$
    and by Lemma \ref{StablePM} applied to $\Uu\restriction \beta+1$ for all $\gamma < \mu$ we have $\pi_{\xi_{\delta},\beta}^{\Uu}(\gamma) < \mu$. 
    Thus for $\gamma = \Upsilon$ we have $\zeta < \pi_{\xi_{\delta},\beta}^{\Tt}(\Upsilon) < \mu$. 
    
    Therefore, $$\mu \geq \sup_{\gamma \in \mu}(\pi^{\Tt}_{\xi_{\delta},\beta}(\gamma)) \geq \pi^{\Tt}_{\xi_{\delta},\beta}(\mu) \geq \mu ,$$
    	and $$\Ult^{*}(\cM^{\Tt}_{\xi_{\delta}^{\Tt}}|\mu,E_{\delta}^{\Tt})= \cM_{\beta}^{\Tt}|\mu.$$
    	
    $\blacktriangleright$ Suppose $D^{\Tt}\cap [0,\beta]_{T} \neq \emptyset$. We need to further subdivide into two subcases depending on whether $D^{\Tt}\cap[0,\xi_{\delta}^{\Tt}]$ is empty or not.
    
    \begin{itemize}
    	\item[$\star$] If $D^{\Tt}\cap [0,\xi_{\delta}^{\Tt}]_{T} = \emptyset$, then $\beta \in D^{\Tt}$ and by induction hypothesis we have $\cM_{\xi_{\delta}^{\Tt}}^{\Tt}|\mu =\cM_{\xi_{\delta}^{\Tt}}^{\Uu}$ and $E_{\delta}^{\Tt}$ is not a total extender on $\cM_{\xi_{\delta}^{\Tt}}^{\Tt}$. 
    	
    	As $\cM_{\xi_{\delta}^{\Tt}}^{\Tt}$ is an acceptable $J$-structure, it follows that 
    	$$ \cM_{\xi_{\delta}^{\Tt}}^{\Tt} \cap H_{\mu} = \cM_{\xi_{\delta}^{\Tt}}^{\Tt}|\mu \supseteq \mathcal{P}(\crit(E_{\delta}^{\Tt}))\cap \cM_{\xi_{\delta}^{\Tt}}^{\Tt}$$ 
    	
    	Hence, if $\eta_{\delta}^{\Tt} \geq \mu$, it follows that $E_{\delta}^{\Tt}$ is a total extender on $\cM_{\xi_{\delta}}^{\Tt}$ and $\beta \not\in D^{\Tt}$, which is a contradiction as we are assuming that $\beta \in D^{\Tt}$.
    	
    	Therefore $\eta_{\delta}^{\Tt} < \mu$, $\eta^{\Tt}_{\delta}=\eta^{\Uu}_{\delta}$ and by our induction hypothesis 
    	$$\cM_{\xi_{\delta}^{\Tt}}^{\Tt}||\eta_{\delta}^{\Tt} = \cM_{\xi_{\delta}^{\Tt}}^{\Uu}||\eta_{\delta}^{\Tt}.$$ 
    	
    	Then 
    	$$ \cM_{\beta}^{\Tt}= \Ult^{*}(\cM_{\xi_{\delta}^{\Tt}}^{\Tt}||\eta^{\Tt}_{\delta},E_{\delta}^{\Tt}) = \Ult^{*}(\cM_{\xi_{\delta}^{\Tt}}^{\Uu}||\eta^{\Uu}_{\delta},E^{\Uu}_{\delta}) = \cM_{\beta}^{\Uu}.$$
    	
    	\item[$\star$] If $D^{\Tt}\cap[0,\xi_{\delta}^{\Tt}]_{T}\neq\emptyset$, then by our induction hypothesis $\cM_{\xi_{\delta}^{\Tt}}^{\Tt} = \cM_{\xi_{\delta}^{\Tt}}^{\Uu}$, so 
    	$$ \cM_{\beta}^{\Tt}= \Ult^{*}(\cM_{\xi_{\delta}^{\Tt}}^{\Tt},E_{\delta}^{\Tt}) = \Ult^{*}(\cM_{\xi_{\delta}^{\Tt}}^{\Uu},E^{\Uu}_{\delta}) = \cM_{\beta}^{\Uu}.$$

    	\end{itemize}

     $\bullet$ Suppose $\beta$ is a limit ordinal. Again, we divide into two subcases based on whether $D^{\Tt}\cap [0,\beta]_{\Tt}$ is empty or not.
     
     $\blacktriangleright$ Suppose $D^{\Tt}\cap[0,\beta]_{T}=\emptyset$. Given $\gamma < \pi_{0,\beta}^{\Tt}(\mu)$, we have that $\zeta < \gamma$ iff there are $\bar{\beta} \in [0, \beta)_{T}$ and $\bar{\zeta} < \bar{\gamma} < \mu = \pi_{0,\bar{\beta}}^{\Tt}(\mu)$ such that $\pi_{\bar{\beta},\beta}^{\Tt}(\bar{\zeta})  = \zeta $ and $\pi_{\bar{\beta},\beta}^{\Tt}(\bar{\gamma})=\gamma$. 
     
     By a cardinality argument it follows that $\gamma <\mu$. Therefore $\pi_{0,\beta}^{\Tt}(\mu)\leq \mu$ and 
    $$ \cM_{\beta}^{\Tt}|\mu = \dirlim_{\bar{\beta}\in[0,\beta]_{T}}(\cM_{\bar{\beta}}^{\Tt}|\mu,\pi_{\bar{\beta},\beta}^{\Tt}|\mu) = \dirlim_{\bar{\beta}\in [0,\beta]_{T}}(\cM_{\bar{\beta}}^{\Uu},\pi_{\bar{\beta},\beta}^{\Uu})=\cM_{\beta}^{\Uu}$$
    
    $\blacktriangleright$ Suppose $D^{\Tt}\cap[0,\beta]_{T}\neq\emptyset$, let $\zeta$ be the largest element in $D^{\Tt}\cap[0,\beta]_{\Tt}$. Then by induction hypothesis 
    $$\cM_{\beta}^{\Tt} = \dirlim_{\bar{\beta}\in[\zeta,\beta]_{T}}(\cM^{\Tt}_{\bar{\beta}},\pi^{\Tt}_{\bar{\beta},\beta})=\dirlim_{\bar{\beta}\in[\zeta,\beta]_{\Uu}}(\cM_{\bar{\beta}}^{\Uu},\pi_{\bar{\beta},\beta}^{\Uu}) = \cM_{\beta}^{\Uu}$$	\end{proof}

\begin{remark}
	Given a proper class premouse $L[E]$ and a regular cardinal $\mu$ such that $L[E]|\mu$ is $\mu$-stable above $\kappa$, our next result, Lemma \ref{Lives}, gives a sufficient condition on iteration trees $\Tt$ on $L[E]$ for $\Tt$ to live on $L[E]|\mu$. 
\end{remark}

\begin{lemma}\label{Lives}
	Let $\mu$ be a regular cardinal. Suppose $\cM$ is a proper class premouse and $\Tt$ is a finite \footnote{This lemma remains true if we drop the hypothesis that $\Tt$ is finite.} iteration tree on $L[E]$ such that for all $\beta+1 \in lh(\Tt)$ we have $\crit(E_{\beta}^{\Tt} ) >\kappa$ and $lh(\Tt) < \mu$.  Suppose further that $L[E]|\mu$ is $\mu$-stable above $\kappa$. 
	If for all $\beta \in lh(\Tt)$ we have $\nu_{\beta}^{\Tt} <\mu$, then $\Tt$ lives on $\cM|\mu$. 
\end{lemma}
\begin{proof} 	We define recursively an iteration tree $\mathcal{U}$ on $\cM|\mu$ such that $\lh(\Uu) \leq \lh(\Tt)$ as follows: If $\Uu|\beta$ is defined, $\beta\in \Tt$, $\cM^{\Uu}_{\beta} \cap \ord \geq \nu_{\beta}^{\Tt}$ and $\cM_{\beta}^{\Uu}||\nu_{\beta}^{\Tt} = \cM_{\beta}^{\Tt}||\nu_{\beta}^{\Tt}$, then we let $\nu_{\beta}^{\Uu}=\nu_{\beta}^{\Tt}$, otherwise we let $\nu_{\beta}^{\Uu}$ be undefined and $lh(\Uu)=\beta$.
	
	Suppose $\beta \in \lh(\Tt)$, $\beta+1 < \lh(\Tt)$ and $\beta < \lh(\Uu)$, we will verify that $\beta+1 < \lh(\Uu)$. 
	
    $\blacktriangleright$ If $D^{\Tt}\cap [0,\beta]_{T}=\emptyset$, by Lemma \ref{IterationBounds} applied to $\Tt\restriction \beta+1$ we have $ \cM_{\beta}^{\Tt}|\mu = \cM_{\beta}^{\Uu}$.
    
    Hence $\nu_{\beta}^{\Tt} < \mu = \cM_{\beta}^{\Uu}\cap \ord $ and $\cM_{\beta}^{\Uu}||\nu_{\beta}^{\Tt}= \cM_{\beta}^{\Tt}||\nu_{\beta}^{\Tt}$.
    
   $\blacktriangleright$ If $D^{\Tt}\cap [0,\beta]_{T} \neq \emptyset $, then by Lemma \ref{IterationBounds} applied to $\Tt\restriction \beta+1$ we have  $\cM_{\beta}^{\Tt}=\cM_{\beta}^{\Uu}$.
    
    Therefore $\nu_{\beta}^{\Tt}\in \cM_{\beta}^{\Uu}\cap \ord$ and  $\cM_{\beta}^{\Uu}||\nu_{\beta}^{\Tt}=\cM_{\beta}^{\Uu}||\nu_{\beta}^{\Tt}$. \end{proof}

\begin{remark}\label{rmk1}\begin{enumerate}
		\item[(a)]	Suppose that $V$ is an extender model $L[E]$, $\lambda$ is a cardinal  and  $\Tt$ is an iteration tree on $L[E]$ such that $sup_{\alpha \in \lh(\Tt)}\nu_{\alpha}^{\Tt} < \lambda^{+}$. Then by Lemma \ref{Lives} for $\mu:= \lambda^{++L[E]}$ we have that $\Tt$ lives on $L[E]|\mu$. Notice that $L[E]|\mu$ and $\Tt$ are in the hypothesis of Lemma \ref{StablePM}. So in particular if $\mathcal{U}= \Tt\restriction (L[E]|\mu)$ is the restriction of $\Tt$ to $L[E]|\mu$, then for all $\beta \in \lh(\mathcal{U})$ it follows that $\cM_{\beta}^{\mathcal{U}} \cap \ord \leq \mu$. 
		\item[(b)] Our hypotheses on Lemma \ref{IterationBounds} are optimal in the following sense: suppose that $\mu$ is a regular cardinal and there is $\gamma <\mu$ such that $$(\gamma \text{ is the largest cardinal} \wedge \cf(\gamma) >\kappa)^{L[E]|\mu}.$$
		If there is $\beta \in \mu$ such that $(crit(E_{\beta})=cf(\gamma))^{L[E]}$ and $E_{\beta}$ is a total measure in $\cf^{L[E]}(\gamma)$, then $Ult_{0}(L[E]|\mu,E_{\beta}) \cap \ord > \mu$. 
		\item[(c)] In Lemma \ref{CES} we show that we can not drop the hypothesis that $L[E]|\mu$ is $\mu$-stable above $\kappa$ in the statement of Theorem A. In the proof of Lemma \ref{CES} we use Remark \ref{rmk1} (b).
	\end{enumerate}
\end{remark}

\begin{lemma}\cite[Lemma 1.1]{SchALCM} \label{IS} Let $\mathcal{M} = \langle J_{\alpha}^{E}, \in, E, F\rangle$ be an iterable premouse, where $ F\neq \emptyset$ . Suppose that for no $\mu \leq \mathcal{M} \cap OR$ do we have \[ \mathcal{J}_{\mu}^{\mathcal{M}} \models ZFC + \ \text{there is a Woodin cardinal}. \] Set $\kappa = crit(F)$ and let $ \xi \in (\kappa,\rho_{1}(\mathcal{M}))$. Then there is some $\tilde{\nu} \in (\xi, \xi^{+\mathcal{M}})$ with $crit(E_{\tilde{\nu}}) = crit(F)$. 
\end{lemma}

\begin{lemma}\label{Oo} Suppose that $V$ is a proper class premouse  $L[E]$ that is iterable. Let  $ \kappa < \mu $ be cardinals. Then 	\begin{gather*}   O(\kappa) > \mu \ \longleftrightarrow \ o(\kappa) > \mu.
	\end{gather*} Moreover, 
	\begin{gather*} O(\kappa) = \mu  \longleftrightarrow \  o(\kappa) = \mu. 
	\end{gather*} 
\end{lemma}
\begin{proof}  Notice that if $X \subseteq \ord$, then $\otp(X) \leq \sup(X)$. From this general fact it follows that $o(\kappa) > \mu$ implies $O(\kappa) > \mu$. For the other direction we split the analysis into two cases. 
	
	$\blacktriangleright$ Suppose first that $ \mu$ is a limit cardinal and  $O(\kappa) > \mu $. We will verify that $o(\kappa) > \mu$. For that we show that given a regular cardinal $\chi < \mu$  such that $\kappa < \chi$ the following holds:	\begin{gather}\label{supI}  sup(\{ \beta < \chi\mid crit(E_{\beta}) =\kappa  \}) = \chi.
	\end{gather} 
	 
	Let $  \alpha > \chi$ be such that $crit(E_{\alpha})=\kappa$ and let $ \mathcal{M} =\mathcal{J}_{\alpha}^{L[E]}$. Then $ \rho_{1}(\mathcal{M}) \geq \chi.$ 
	
	Given $ \xi < \chi$, by Lemma \ref{IS} there is $ \tilde{\xi}  \in (\xi,\xi^{+\mathcal{M}}) $ such that $ crit(E^{\mathcal{M}}_{\tilde{\xi}}) = \kappa $. As $ \xi $ was arbitrary the equality in \eqref{supI} follows. 
	
	Thus \eqref{supI} holds for any regular cardinal $ \chi $ such that $\kappa < \chi < \mu$. Therefore
	\begin{gather*} |\{ \beta < \mu \mid  crit(E_{\beta}) = \kappa \}| = \mu. 
	\end{gather*}  
	
    Then  $ \otp(\{ \beta < \mu \mid  crit(E_{\beta}) = \kappa \}) \geq \mu$.  Notice that 
    \begin{gather*} o(\kappa) =\otp(\{\beta < \mu \mid \crit(E_{\beta})=\kappa\} \cup \{\beta \geq \mu \mid \crit(E_{\beta})=\kappa\} ) =  \\ \otp(\{\beta < \mu \mid \crit(E_{\beta})=\kappa\}) \oplus \otp(\{\beta \geq \mu \mid \crit(E_{\beta})=\kappa\}) \end{gather*} where $\oplus$ represents the ordinal sum. As $O(\kappa)> \mu$ we have $$\otp(\{ \beta \geq \mu \mid  crit(E_{\beta}) = \kappa \}) > 0 $$ Therefore $o (\kappa) \geq \mu+1 > \mu$.

	$\blacktriangleright$ Suppose $\mu = \theta^{+}$ and suppose $O(\kappa)> \mu$. Fix $\alpha$ such that $ \alpha > \mu $ with $\crit(E_{\alpha}^{\cM}) = \kappa$ and consider $\mathcal{M} = \mathcal{J}_{\alpha}^{L[E]}$. We have $\rho_{1}(\mathcal{M}) \geq \mu$ and given $\xi \in (\kappa,\mu)$, by Lemma \ref{IS}, there is $\tilde{\xi} \in (\xi,\xi^{+\mathcal{M}})$ such that $ crit( E_{\tilde{\xi}}^{\mathcal{M}}))= \kappa$. Hence 
	\begin{gather} \label{SupMu} sup( \{ \beta < \mu \mid crit(E_{\beta}) = \kappa \}) = \mu. 
	\end{gather}  Therefore  \eqref{SupMu} with $O(\kappa) > \mu$ implies that $o(\kappa) > \mu$. 
	
	For the second part, suppose $o(\kappa) = \mu$, then  $O(\kappa) \leq \mu$, otherwise by the first part we would have $ o(\kappa) > \mu$. Hence $O(\kappa) =\mu$. For the other direction  again we split the analysis into two cases:
	
	$\blacktriangleright$ Suppose $ O(\kappa)=\mu$ and $\mu$ is a limit cardinal. Given $\chi$ a regular cardinal, such that $\kappa < \chi < \mu$ we have  by the first part of the lemma that $o(\kappa) > \chi$. Hence $o(\kappa) \geq \mu$ and thus $o(\kappa) = \mu$.  
	
	$\blacktriangleright$ Suppose $O(\kappa)=\mu$ and $ \mu = \theta^{+}$ for some cardinal $\theta$. Since $\mu$ is a regular cardinal it follows that  $o(\kappa) = \mu$.  
\end{proof}

\begin{lemma} \label{gap} Suppose $V$ is an extender model $L[E]$ which is iterable. Let $ \kappa$ and $\mu$ be cardinals, such that $\kappa < 
	\mu$. Suppose that $O(\kappa) \leq \mu$, $\{ \alpha < \kappa \mid O(\alpha) > \mu \} $ is bounded in $\kappa$ and $\kappa < O(\kappa)$ \footnote{For example, when there exists a total measure indexed on $E$ with critical point $\kappa$ we have $\kappa^{+} < O(\kappa)$.} . Let
	\begin{gather}
	\beta^{*} = sup \{ \ \alpha < \kappa  \mid O(\alpha) >\mu \ \}.
	\end{gather}
	and 
	\begin{gather}
	\Theta = sup \{ \ O(\alpha) \mid  \beta^{*} < \alpha \leq  \kappa  \ \} \leq \mu
	\end{gather}
	
	Then there is no $\eta \in (\beta^{*},\Theta] $ such that $O(\eta) > \Theta$.
	
\end{lemma}

\begin{proof} Suppose otherwise.  Let $ E_{\alpha} $ be such that $ \alpha> \Theta $ and $ \eta:=crit(E_{\alpha}) \in (\beta^{*},\Theta]$ and let $\mathcal{M}^{*}$ be the largest initial segment of $L[E]$ where we can apply $E_{\alpha}$. Let $\cN:= Ult_{0}(\cM^{*},E_{\alpha})$ and $\pi:\cM^{*} \rightarrow \cN$ the ultrapower map. 
	We split the analysis into two cases and the second case will split into two subcases.
	
	$\bullet$ Suppose  $\eta = \Theta$. In this case\footnote{See Figure \ref{figure:Diagram1}.}  $$L[E]|\alpha\models ``\Theta \text{ is a cardinal}"$$ and hence $E_{\Theta}=\emptyset$, as extenders are not indexed at cardinals. Therefore  $$\Theta = sup (\{O(\alpha) \mid \beta^{*} < \alpha \leq \kappa \} \cap \Theta)$$  and $\Theta$ is a limit of ordinals $\gamma$ such that $\crit(E_{\gamma}) \in (\beta^{*},\kappa]$.
	
	Let $\varphi(\Theta,\beta^{*},\kappa)$ be the following formula:
	$$\forall \gamma < \Theta  ~ \exists \gamma' < \Theta ~ ( ~ \gamma < \gamma' ~ \wedge ~ \dot{E}_{\gamma'} \neq \emptyset ~ \wedge ~ crit(\dot{E}_{\gamma}) \in (\beta^{*},\kappa]  ~).$$
	Note that $\varphi(\Theta,\beta^{*},\kappa)$ is a $\Sigma_{0}$-formula in the language $\{\in,\dot{E}\}$.

	We have \begin{gather*} L[E] \models\varphi(\Theta,\beta^{*},\kappa),
	\end{gather*}  
	and therefore
	\begin{gather*} \mathcal{M}^{*} \models \varphi(\Theta,\beta^{*},\kappa).
	\end{gather*}

		Notice that $O(\kappa) > \kappa$ implies $\Theta \geq O(\kappa) > \kappa$. Hence by $\Sigma_{1}$-elementarity of $\pi_{E_{\alpha}}$ we have:
		\begin{gather*}  \cN \models \varphi(\pi_{E_{\alpha}}(\Theta),\underbrace{\beta^{*}}_{=\pi_{E_{\alpha}}(\beta^{*})},\underbrace{\kappa}_{ = \pi_{E_{\alpha}}(\kappa)}).
		\end{gather*}
		
		 Thus, in $\cN$, $ \pi_{E_{\alpha}}(\Theta)$ is a limit of indexes of extenders with critical points in the interval $ (\beta^{*},\kappa]$. Notice that as $\crit(E_{\alpha})= \Theta$ then $\Theta < \pi_{E_{\alpha}}(\Theta)$. Therefore there is $\gamma$ such that 
		 $\Theta < \gamma < \pi_{E_{\alpha}}(\Theta)$  and $\crit(E^{\cN}_{\gamma}) \in (\beta^{*},\kappa]$. 
		 
		 As $\cN|\alpha= \cM^{*} | \alpha$ and $\alpha= \pi_{E_{\alpha}}(\Theta)^{+\cN}$, it follows that $E_{\gamma}^{\cN}=E_{\gamma}$. 
		 But $\Theta < \gamma $ and $\crit(E_{\gamma}) \in (\beta^{*},\kappa]$ contradict the definition of $\Theta$.
		 
	
	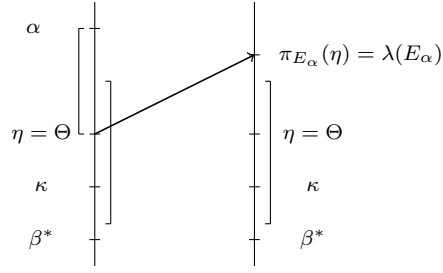
\begin{figure*} 
		\centering
		\begin{tikzpicture}[scale=0.7]
		
		\draw (0.9,0)--(1.1,0) node {};
		\draw (-0.15,0) node {\footnotesize{$\alpha$}};
		\draw (0,-2) node {\footnotesize{$\eta = \Theta$}};
		\draw (0.9,-2)--(1.1,-2) node {};
		\draw (0.8,0)--(0.7,0);
		\draw (0.7,-2)--(0.7,0);
		\draw (0.8,-2)--(0.7,-2);
		\draw (1,-2) node{};
		
		\draw (5.1,-2) node {\footnotesize{$\eta=\Theta$}};
		\draw (3.9,-2)--(4.1,-2) node {};
		\draw (5,-2) node{};
		\draw (6,-0.5) node {\footnotesize{$\pi_{E_{\alpha}}(\eta) = \lambda(E_{\alpha}) $}};
		\draw (3.9,-0.5)--(4.1,-0.5) node {};
		\draw (5.1,1) node{};
		\draw (0,-3) node {\footnotesize{$\kappa$}};
		\draw (3.9,-3)--(4.1,-3) node {};
		\draw (3.9,-4)--(4.1,-4) node {};
		\draw (0.9,-3)--(1.1,-3) node {};
		\draw (0.9,-4)--(1.1,-4) node {};
		\draw (5.1,-3) node{\footnotesize{$\kappa$}};
		\draw (0,-4) node{\footnotesize{$\beta^{*}$}};
		\draw (5.1,-4) node{\footnotesize{$\beta^{*}$}};
		\draw (1,0.5)--(1,-4.5);
		\draw (4,0.5)--(4,-4.5);
		
		\draw  (1.2,-1)--(1.3,-1);
		\draw  (1.3,-1)--(1.3,-3.7);
		\draw  (1.2,-3.7)--(1.3,-3.7);
		\draw  (4.2,-1)--(4.3,-1);
		\draw  (4.3,-1)--(4.3,-3.7);
		\draw  (4.2,-3.7)--(4.3,-3.7);
		\draw[->,semithick] (1,-2)--(4,-0.5);	\end{tikzpicture}
		\caption{\footnotesize{Case 1, Lemma \ref{gap}}} \label{figure:Diagram1}
	\end{figure*}
	
	$\bullet$ Suppose  $\eta < \Theta  $. 
	We split this case into two subcases (see Figures \ref{figure:Diagram2} and \ref{figure:Diagram3}):
	
	$\blacktriangleright $ Suppose $\eta < \Theta  $ and $ \lambda(E_{\alpha}) \leq \Theta$. 
	Then \begin{gather*}\pi_{E_{\alpha}}(\Theta) \geq \pi_{E_{\alpha}}(\lambda(E_{\alpha})) > \alpha,
	\end{gather*} and by $\Sigma_{0}$-elementarity there is a $\gamma \in (\alpha,\pi_{E_{\alpha}}(\Theta)) $ such that $E_{\gamma}^{\cN} \neq \emptyset$ and	\begin{gather*} crit(E_{\gamma}^{\cN}) \in (\beta^{*},\kappa]. 
	\end{gather*} 
		From $	 (\alpha\ \text{is a cardinal})^{\cN}$,
	it follows that $	\rho_{1}(\mathcal{J}^{\cN}_{\gamma}) \geq \alpha$ 	for any $\gamma \in (\alpha,\pi_{E_{\alpha}}(\Theta) ) $.
	
	Since $\alpha = \lambda(E_{\alpha})^{+\cN}$ and $\alpha > \Theta$, we have that \begin{gather*}\rho_{1}(\cN||\gamma) \geq \alpha = \Theta^{+\cN} > \Theta,
	\end{gather*} it follows by Lemma \ref{IS}  that there is $\gamma' \in (\Theta, \Theta^{+\cN}) $ such that  $E_{\gamma'}^{\cN} \neq \emptyset$ and $
	crit(E_{\gamma'}^{\cN}) = crit(E_{\gamma}^{\cN})$.	
	As $\Theta^{+\cN} \leq \alpha$ and  $\cM^{*}|\alpha= \cN|\alpha$, we have $E_{\gamma'}^{\cN}=E_{\gamma'}$, which contradicts the definition of $\Theta$. 
	
	\begin{figure*} 
		\centering
		\begin{tikzpicture}[scale=0.7]
		\draw (0,0.5) node {\footnotesize{$\alpha$}};
		\draw (0,-0.5) node {\footnotesize{$\Theta$}};
		\draw (0.9,-3)--(1.1,-3) node {};
		\draw (0.9,-2.5)--(1.1,-2.5) node {};
		\draw (0.8,0.5)--(0.7,0.5);
		\draw (0.7,-2.5)--(0.7,0.5);
		\draw (0.8,-2.5)--(0.7,-2.5);
		\draw (0,-2.5) node{\footnotesize{$ \eta $}};
		\draw (6,0.5) node{\footnotesize{$\Theta^{+} =\alpha \leq \rho_{1}(\cN||\gamma) $}};
		\draw (3.9,1)--(4.1,1) node{};
		\draw (4.5,1) node{\footnotesize{$\gamma$}};
		\draw (3.9,0.5) -- (4.1,0.5);
		\draw (3.9,-3)--(4.1,-3);
	    \draw (6.8,-0.5) node {\footnotesize{$ \lambda(E_{\alpha})= \pi_{E_{\alpha}}(\eta) \leq \Theta $}};
		\draw (5,-2) node{};
		\draw (5.1,1) node{};
		\draw (0.5,-3) node {\footnotesize{$\kappa$}};
		\draw (3.9,-0.5)--(4.1,-0.5) node {};
		\draw (3.9,-4)--(4.1,-4) node {};
		\draw (0.9,-0.5)--(1.1,-0.5) node {};
		\draw (0.9,-4)--(1.1,-4) node {};
		\draw (4.5,-3) node{\footnotesize{$\kappa$}};
		\draw (0.5,-4) node{\footnotesize{$\beta^{*}$}};
		\draw (4.5,-4) node{\footnotesize{$\beta^{*}$}};
		\draw (1,2)--(1,-4.5);
		\draw (4,2)--(4,-4.5);
		\draw (5.0,1.7) node{\footnotesize{$\pi_{E_{\alpha}}(\Theta)$}};
		\draw (3.9,1.7)--(4.1,1.7);
		\draw  (1.2,0)--(1.3,0);
		\draw  (1.3,0)--(1.3,-3.7);
		\draw  (1.2,-3.7)--(1.3,-3.7);
		\draw  (4.2,0)--(4.3,0);
		\draw  (4.3,0)--(4.3,-3.7);
		\draw  (4.2,-3.7)--(4.3,-3.7);
	    \draw[->] (1,-2.5)--(4,-0.5);
		\draw[->] (1,-0.5)--(4,1.7);
		\end{tikzpicture}
		\caption{\footnotesize{Case 2 (i), Lemma \ref{gap} }} \label{figure:Diagram2}
	\end{figure*}
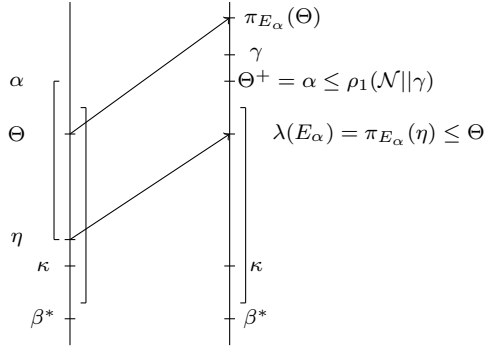

	$\blacktriangleright $ Suppose $\eta < \Theta$ and   $\lambda(E_{\alpha}) > \Theta$.
	In this case $\pi_{E_{\alpha}}(\Theta) > \lambda(E_{\alpha}) > \Theta $. Then  for all $\gamma$ such that
	$\gamma \in (\lambda(E_{\alpha}), \pi_{E_{\alpha}}(\Theta))$ we have 
	\begin{gather}
	\rho_{1}(\mathcal{J}^{\cN}_{\gamma}) \geq \lambda(E_{\alpha}) > \Theta.
	\end{gather}
	
	Then like in case 2 (i) we can find an extender in the sequence of $L[E]$ that is indexed in the interval $(\Theta, \lambda(E_{\alpha}))$ with critical point in the interval $(\beta^{*},\kappa]$, contradicting the definition of $\Theta$.  
	
	\begin{center}
		\begin{figure*} 
			\centering
			\begin{tikzpicture}[scale=0.5] 
		
			\draw (-0.3,0.8) node {\footnotesize{$\alpha$}};
			\draw (0,-1.5) node {\footnotesize{$\Theta$}};
			\draw (5,-1.5) node {\footnotesize{$\Theta$}};
			\draw (3.9,-1.5)--(4.1,-1.5);
			\draw (0.9,-3)--(1.1,-3) node {};
			\draw (0.9,-2.5)--(1.1,-2.5) node {};
			\draw (0.8,0.8)--(0.7,0.8);
			\draw (0.7,-2.5)--(0.7,0.8);
			\draw (0.8,-2.5)--(0.7,-2.5);
			\draw (0,-2.5) node{\footnotesize{$ \eta $}};
			\draw (3.9,0.7)--(4.1,0.7) node{};
			\draw (4.5,0.7) node{\footnotesize{$\gamma$}};
	    	\draw (3.9,-3)--(4.1,-3);
			\draw (7,-0.2) node {\footnotesize{$ \lambda(E_{\alpha})= \pi_{E_{\alpha}}(\eta) > \Theta $}};
			\draw (5,-2) node{};
			\draw (5.1,1) node{};
			\draw (0.5,-3) node {\footnotesize{$\kappa$}};
			\draw (3.9,-0.2)--(4.1,-0.2) node {};
			\draw (3.9,-4)--(4.1,-4) node {};
			\draw (0.9,-1.5)--(1.1,-1.5) node {};
			\draw (0.9,-4)--(1.1,-4) node {};
			\draw (4.7,-3) node{\footnotesize{$\kappa$}};
			\draw (0.5,-4) node{\footnotesize{$\beta^{*}$}};
			\draw (4.5,-4) node{\footnotesize{$\beta^{*}$}};
			\draw (1,2)--(1,-4.3);
			\draw (4,2)--(4,-4.3);
			\draw (5.2,1.5) node{\footnotesize{$\pi_{E_{\alpha}}(\Theta)$}};
			\draw (3.9,1.5)--(4.1,1.5);
			\draw  (1.2,-1)--(1.3,-1);
			\draw  (1.3,-1)--(1.3,-3.7);
			\draw  (1.2,-3.7)--(1.3,-3.7);
			\draw  (4.2,-1)--(4.3,-1);
			\draw  (4.3,-1)--(4.3,-3.7);
			\draw  (4.2,-3.7)--(4.3,-3.7);
			\draw[->] (1,-2.5)--(4,-0.2);
			\draw[->] (1,-1.5)--(4,1.5);
			\end{tikzpicture}
			
			\caption{\footnotesize{Case 2 (ii), Lemma \ref{gap}}} \label{figure:Diagram3}
		\end{figure*}
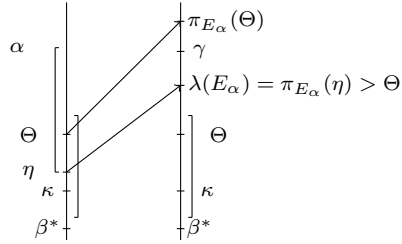
	\end{center}
\end{proof}

\section{Equivalence \label{Equivalence}}

In this section we prove Theorem A and Theorem B. We will need some results from core model theory before we start with the proofs of  Theorems A and B.

\begin{remark}
	Suppose there is no inner model with a Woodin cardinal and let $\cK$ be the core model. If $\kappa$ is an ordinal and $\Xi$ is a singular strong limit cardinal above $\kappa$, then by  (iii) of Theorem \ref{Knm} $$\cK|\Xi^{+}\models ``\Xi \text{ is the largest cardinal}".$$ By Definition~\ref{Stability} we have:
		\begin{enumerate}
			\item[(a)] $\cK|\Xi^{+}$ is $\Xi^{+}$-stable above $\kappa$, or
			\item[(b)]  there is $\beta \in \Xi^{+}$ such that $\beta$ is the least ordinal such that $E_{\beta}^{\cK}$ is a total measure over $\cK|\Xi^{+}$ with $\crit(E_{\beta}^{\cK}) = \cf^{\cK}(\Xi)$ and $cf^{\cK}(\Xi) \geq \kappa$.
		\end{enumerate}
\end{remark}

\begin{definition}
	Suppose there is no inner model with a Woodin cardinal and let $\cK$ be the core model. If $\kappa$ is an ordinal and $\Xi$ is a singular strong limit cardinal above $\kappa$ we say that $\cW$ is the \emph{ $(\Xi^{+},\kappa)$-stabilization of $\cK$} iff 
	\begin{enumerate}
		\item[(a)] $\cW = \cK|\Xi^{+}$ and $\cK|\Xi^{+}$ is $\Xi^{+}$-stable above $\kappa$, or
		\item[(b)] $\cW = \Ult_{0}(\cK,E_{\beta}^{\cK})|\Xi^{+}$ where $\beta$ is the least measure indexed in the sequence of $\cK$ that is total in $\cK|\Xi^{+}$ with $\crit(E_{\beta}^{\cK})=\cf^{\cK}(\Xi)$ and $\cf^{\cK}(\Xi)\geq\kappa$.
	\end{enumerate} 
\end{definition}

\begin{definition}
	Let $\cM$ be a premouse and $ \kappa$ an ordinal such that $\kappa \in \cM$. We say that $\kappa$ is a \emph{ strong cutpoint of $\cM$} iff for all $\alpha \leq \cM\cap \ord$  such that $\alpha > \kappa$ we have that either $E_{\alpha}^{\cM}=\emptyset$ or $\crit(E_{\alpha}^{\cM})> \kappa$.
\end{definition}

The following is a slight variation of \cite[Proposition 4.4]{Knm}.
\begin{lemma}{\cite[Proposition 4.4]{Knm}} \label{Universal} Let $\Omega$ be a regular cardinal and $\kappa \in \Omega$. Suppose that $\cW$ is  $\Omega$-stable above $\kappa$, $(\Omega+1)$-iterable, $\kappa$ is a strong cutpoint of $\cW$ and $
	\cW $ has a largest cardinal. 
		Then for every sound premouse $\cM$ such that:
	\begin{itemize} 
		\item $\cM$ is $(\Omega+1)$-iterable,
		\item  $\cM \cap \ord < \Omega $,
	    \item $\cM||\kappa = \cW||\kappa$, 
	    \item $\kappa$ is a strong cutpoint of $\cM$,
	\end{itemize}
we have that there are\footnote{These iteration trees are obtained by the so called Comparison Lemma, see \cite[Section 3.2]{MR2768698} or \cite[Lemma 9.1.8]{MR1876087}.} iteration trees $\Tt$ and $\Uu$ on $\cW $ and $\cM$ respectively such that 
\begin{enumerate}
		\item[(a)] for $b^{\Uu}$ the main branch of $\Uu$ we have $D^{\Uu}\cap b^{\Uu}=\emptyset$, 
		\item[(b)] $\cM_{\infty}^{\Uu}$ is sound,
		\item[(c)] $\cM_{\infty}^{\Uu} \triangleleft \cM_{\infty}^{\Tt}$,
		\item[(d)] for every $\alpha \in \Tt$ and for every $\beta \in \Uu$ we have $\nu_{\alpha}^{\Tt},\nu_{\beta}^{\Uu} > \kappa$.	
\end{enumerate}
\end{lemma} 
	The following lemma is standard but we include it for the reader's convenience.
\begin{lemma} \label{ISK}
	Suppose there is no inner model with a Woodin cardinal and let $\cK$ be the core model. Let $\kappa$ be an ordinal and $\cM$ a sound iterable premouse. Suppose 
	\begin{itemize}
		\item $\kappa$ is a strong cutpoint of $\cK$ and $\cM$,
		\item $\cM||\kappa = \cK||\kappa$, and
		\item $\rho_{\omega}(\cM)\leq \kappa$.
	\end{itemize}	
	Then $\cM \triangleleft \cK$.
\end{lemma}
\begin{proof} Let $\kappa$ be as in the hypotheses of the lemma and let $\Xi$ be a singular strong limit cardinal such that $\cf(\Xi) > \kappa$. Let $\cW$ be the $(\Xi^{+},\kappa)$-stabilization of $\cK|\Xi^{+}$. 
	We can apply  Lemma \ref{Universal} to $\cM$, $\cW$ and $\kappa$. Let $\Tt$ and $\Uu$ be iteration trees given by  Lemma \ref{Universal} where $\Tt$ is on $\cW$ and $\Uu$ is on $\cM$. 
	
	\begin{claim} \label{CutPoint} For all $\alpha \in \lh(\Uu)$ we have $\crit(E_{\alpha}^{\Uu}) \geq \kappa$ and $\kappa$ is a strong cutpoint of $\cM_{\alpha}^{\Uu}$.
		\end{claim} 
	\begin{proof} We prove the claim by induction on  $\alpha \in \lh(\Uu)$. 
		Suppose that $\alpha \in \lh(\Uu)$ and that for all $\beta < \alpha$ we have that $\kappa$ is a strong cutpoint of $\cM_{\beta}^{\Uu}$ and $\crit(E_{\beta}^{\Uu})>\kappa$.
		
		$\blacktriangleright$ Suppose $\alpha$ is a limit ordinal. Let $\gamma \in [0,\alpha]_{T^{\Uu}}$ be large enough such that $D^{\Uu}\cap(\gamma,\alpha]_{T^{\Uu}}=\emptyset$ and hence $\dom(\pi_{\gamma,\alpha}^{\Uu})=\cM_{\gamma}^{\Uu}$, so $\pi_{\gamma,\alpha}^{\Uu}$ is not a partial map but has full domain.
		
		As $ \kappa < \crit(E_{\beta}^{\Uu}) $ for all $\beta < \alpha$, it follows that  $\pi_{\alpha,\beta}^{\Uu}(\kappa)=\kappa$. As  $\kappa$ is a strong cutpoint of $\cM_{\gamma}^{\Tt}$  it follows, by the $\Sigma_{1}$-elementarity of $\pi_{\gamma,\alpha}^{\Uu}$, that $\kappa$ is a strong cutpoint of $\cM_{\alpha}^{\Uu}$. Since $\nu_{\alpha}^{\Uu} > \kappa$, it then follows that $\crit(E_{\alpha}^{\Uu}) > \kappa$.
		
		$\blacktriangleright$ Suppose $\alpha = \gamma+1$ for some $\gamma \in \ord$. From our induction hypothesis we have that $\kappa$ is a strong cutpoint of $\cM_{\xi_{\alpha}^{\Uu}}^{\Uu}$, therefore $\kappa$ is a strong cutpoint of $\cM_{\xi_{\alpha}}^{\Uu}||\eta_{\alpha}^{\Uu}$. By the $\Sigma_{1}$-elementarity of $\pi_{\xi_{\alpha},\alpha}^{\Uu}\restriction (\cM_{\xi_{\alpha}}^{\Uu}||\eta_{\alpha}^{\Uu})$ it follows that $\kappa$ is a strong cutpoint of $\cM_{\alpha}^{\Uu}$. As $\nu_{\alpha}^{\Uu}> \kappa$, it follows that $\crit(E_{\alpha}^{\Uu})>\kappa$.
		\end{proof}

	\begin{claim} \label{Claim2}
		$\cM=\cM^{\Uu}_{\infty}$.
	\end{claim}
\begin{proof} Suppose not, and let $E_{\alpha}^{\Uu}$ be the first extender applied to $\cM_{0}^{\Uu}$ such that $\alpha+1 \in b^{\Uu}$, where $b^{\Uu}$ is the main branch of the iteration tree $\Uu$. By Claim \ref{CutPoint} it follows that $\crit(E_{\alpha}^{\Uu}) > \kappa$. 
	Let $n \in \omega$ be the least  such that $\rho_{n}(\cM) \leq \kappa$. By standard arguments $\rho_{n}(\cM_{\alpha+1}^{\Uu}) = \kappa$, but $$\crit(E_{\alpha}^{\Tt})\not\in \tilde{h}^{n}_{\cM_{\alpha}^{\Uu}}(\crit(E_{\alpha})\cup \{\pi_{0,\alpha+1}^{\Uu}(p_{n}^{\cM_{\alpha}^{\Uu}})\}) \supseteq \tilde{h}^{n}_{\cM_{\alpha}^{\Uu}}(\kappa \cup \{\pi_{0,\alpha+1}^{\Uu}(p_{n}^{\cM_{\alpha}^{\Uu}})\}).$$ Therefore $\cM_{\alpha+1}^{\Uu}$ is not $n$-sound, which contradicts the fact that every element in $b^{\Uu}$ is sound. Hence $\Uu$ is trivial and $\cM=\cM_{\infty}^{\Uu}$.  
	\end{proof}
	
	Notice that if $D^{\Tt}\cap b^{\Tt}=\emptyset$ then $\cW \cap \ord \geq \Xi^{+}$ and if $D^{\Tt}\cap b^{\Tt}\neq \emptyset$ by arguments similar to the proof of Claim \ref{Claim2} it follows that $\cM_{\infty}^{\Tt} $ is not sound. In both cases we cannot have $\cM=\cM^{\Tt}_{\infty}$ as $\cM\cap \ord < \Xi$ and $\cM$ is sound.
	
	Therefore $\cM$ is a proper initial segment of $\cM_{\infty}^{\Tt}$.

	\begin{claim}
		$\cW= \cM_{\infty}^{\Tt}$.
	\end{claim}
	\begin{proof}	

	Suppose $E_{0}^{\Uu}\neq \emptyset$. Then $\nu_{0}^{\Tt}\leq \cM \cap \ord$ and $\nu_{0}^{\Tt}$ is a successor cardinal in $\cM_{\infty}^{\Tt}$. On the other hand $\cM=\tilde{h}^{n}_{\cM}(\kappa\cup\{p_{n}^{\cM}\}) \triangleleft \cM_{\infty}^{\Tt}$ and as $\cM$ is a proper initial segment of $\cM_{\infty}^{\Tt}$ it follows that $$\cM_{\infty}^{\Tt} \models |\cM \cap \ord | \leq \kappa, $$
	contradicting that $$\cM_{\infty}^{\Tt}\models ``\nu_{0}^{\Tt}>\kappa \text{ and } \nu_{0}^{\Tt} \text{is a cardinal}".$$
	
	Thus, $\Tt$ must be trivial and $\cM_{\infty}^{\Tt}=\cW$. 	
		
	\end{proof}

If $\cW = \cK|\Xi^{+}$ we are done, so suppose 	$\cW= \Ult_{0}(\cK,E^{\cK}_{\beta})|\Xi^{+}$ where $E_{\beta}^{\cK}$ is the least total measure indexed on the sequence $E^{\cK}$ with critical point $cf^{\cW}(\Xi) \geq \kappa$. 

As $\cM $ is a proper initial segment of $\cW$ and
\[\cW \models ``|\cM\cap \ord|\leq \kappa\text{ and } \beta \text{ is a cardinal}",\] it follows that $\cM\cap \ord < \beta$. As $\cW|\beta=\cK|\beta$, it follows that $\cM \triangleleft \cK$.

\end{proof}

\begin{remark}
	 Theorem \ref{SK} implies Lemma \ref{Universal} for $\kappa > \aleph_{2}^{V}$. We will use Lemma \ref{Universal} and Theorem \ref{SK} to prove Lemma \ref{lemma}.  
\end{remark}
\begin{thm}  \cite[Lemma 3.5]{Pcf+Woodin} \label{SK} If there is no inner model with a Woodin cardinal, $\cK$ is the core model and $\kappa \geq \aleph_{2}^{V}$ is a $\cK$-cardinal, then for every sound iterable mouse $\mathcal{M} $ such that $ \mathcal{M}||\kappa = \cK||\kappa $ and $\rho_{\omega}(\mathcal{M}) \leq \kappa $ it holds that 
\begin{gather*} \mathcal{M} \triangleleft \cK.
\end{gather*} 
\end{thm}

\begin{definition} We define the following hypothesis:
	\begin{gather*}(\Delta) \longleftrightarrow \Big(  \big(\text{there is no inner model with a Woodin cardinal}\big) \\ \wedge \\ \big( V=L[E] \big)  \ \wedge \ \big( L[E] \text{ is iterable}    \big) \Big)  
	\end{gather*}
\end{definition}

\begin{lemma}[Steel] \label{lemma} Assume $(\Delta)$. Then $V = \cK$.
\end{lemma} 
\begin{proof} By Theorem \ref{Knm}  we can build $\cK$, the core model. We prove by induction on the cardinals $\kappa$ of $V$ that $\cK||\kappa = L[E]||\kappa$.

 \begin{claim} \label{claimI} $\cK||\aleph_{2} = L[E]||\aleph_{2} $
\end{claim}
\begin{proof}  Because of acceptability and soundness there are cofinally many $\alpha < \omega_{1}$ such that $ \rho_{\omega}(L[E]||\alpha)=\omega$. Fix such an $\alpha < \omega_{1}$, and let $\cM=L[E]||\alpha$. We have that $\cM$ is a sound iterable premouse such that $\cM||\omega = \cK||\omega$.  Hence, by Lemma \ref{ISK}  it follows that $\cM\triangleleft \cK$. 
Thus 
\begin{gather*}
\cK|\aleph_{1} = L[E]|\aleph_{1}.
\end{gather*} 

Again by acceptability and soundness there are unboundedly many $\beta < \aleph_{2}$ such that $\rho_{\omega}(L[E]||\beta) = \omega_{1}$. We fix such a $\beta$ and consider $\cN= L[E]||\beta$.
As $\cK|\omega_{1} = L[E]|\omega_{1}$, it follows that $\cN|\omega_{1}=\cK|\omega_{1}$ and as $\omega_{1}$ is a cardinal it follows that $\cN||\omega_{1}=\cK||\omega_{1}$.

\begin{subclaim}\label{Omega1CUTPOINT} $\omega_{1}$ is a strong cutpoint of $L[E]$ and $\cK$.
\end{subclaim}\begin{proof} We start by verifying this for $L[E]$. Suppose $\gamma >\omega_{1}$ and  $E_{\gamma}\neq \emptyset$. Then $\omega_{1}^{L[E]||\gamma}=\omega_{1}$. From $(\crit(E_{\gamma})  \text{ is a limit cardinal} )^{L[E]||\gamma}$ it follows that $\crit(E_{\gamma}) > \omega_{1}$.

	Next we verify it for $\cK$. From Claim \ref{claimI} we have $\cK||\omega_{1}=L[E]||\omega_{1}$. Therefore $\omega_{1}^{\cK}=\omega_{1}$.  Suppose $\gamma > \omega_{1}$ and $\crit(E_{\gamma}^{\cK}) \neq \emptyset$. Then $\omega_{1}^{\cK||\gamma}=\omega_{1}$. From $(\crit(E_{\gamma}^{\cK})  \text{ is a limit cardinal} )^{\cK||\gamma}$ it follows that $\crit(E^{\cK}_{\gamma}) > \omega_{1}$.	
\end{proof}

Thus, by Lemma \ref{ISK} it follows that $\cN \triangleleft \cK$. Therefore we have $\cK|\aleph_{2}=\cM|\aleph_{2}$. Since extenders are not indexed at cardinals we have $E^{\cK}_{\aleph_{2}}=\emptyset=E_{\aleph_{2}}$. Then $\cK||\aleph_{2}=L[E]||\aleph_{2}$. 
\end{proof}

Now suppose $\kappa > \aleph_{2}$ is a successor cardinal in $V$, say $\kappa=\mu^{+} $ and $ \cK|| \mu = L[E]||\mu $. Then by Theorem \ref{SK}
for every $\xi \in (\mu, \kappa)  $ such that $ \rho_{\omega}(L[E]||\xi) \leq \mu $ we have $ L[E]||\xi \triangleleft \cK$. Thus as there are unboundedly many such $\xi$ below $\kappa$ it follows that $ \cK|\kappa = L[E]|\kappa$. As $E_{\kappa}=\emptyset=E^{\cK}_{\kappa}$ it follows that $\cK||\kappa = L[E]||\kappa$. 

Lastly, suppose $\kappa $ is a limit cardinal in $V$ and for every  cardinal $\mu < \kappa $ we have $ L[E]|\mu = \cK|\mu$, then $L[E]|\kappa = \cK|\kappa$. As $\kappa$ is a cardinal, it follows that $E_{\kappa}=\emptyset=E^{\cK}_{\kappa}$. Therefore $\cK||\kappa = L[E]||\kappa$.

This concludes the induction and verifies the lemma.
 \end{proof}

\begin{thm} \cite[Theorem 2.1]{VM} \label{VM} If there is no inner model\footnote{We have omitted the hypothesis that $\mathcal{P}(\mathbb{R}) \subseteq M$ since we start from the hypothesis that there is no inner model with a Woodin cardinal which is stronger than the hypothesis in \cite[Theorem 2.1]{VM}.} with a Woodin cardinal and $ j:V \rightarrow M $ is an elementary embedding and $M^{\omega} \subseteq M$, then there is an iteration tree $\mathcal{T}$ on $ \cK^{V}$ which does not drop along the main branch such that $ \mathcal{M}_{\infty}^{\mathcal{T}} = \cK^{M}$ and $j|\cK = \pi_{0,\infty}^{\mathcal{T}}$.
\end{thm}

\begin{definition}\label{it_induced} Suppose there is no inner model with a Woodin cardinal. Given $j: V \longrightarrow M$ an elementary embedding, let $\mathcal{T}$ and $\mathcal{U}$ be the iteration trees obtained by comparing\footnote{Given two iterable premice $\cM$ and $\cN$ the comparison  between (or coiteration of) $\cM$ and $\cN$ is the process of iterating $\cM$ and $\cN$ so that at each sucessor step $\alpha +1$ the extenders $E_{\alpha+1}^{\Tt}$ and $E_{\alpha+1}^{\Uu}$ are the least extender in the sequences of $\cM_{\alpha}^{\Tt}$ and $\cM_{\alpha}^{\Uu}$ where there is a disagreement. The last models of these iterations should line up, i.e., $\cM_{\infty}^{\Tt} \triangleright \cM_{\infty}^{\Uu}$ or vice versa. See \cite[Section 3.2]{MR2768698} or \cite[Lemma 9.1.8]{MR1876087} .} $\cK^{V}$ and $\cK^{M} $ respectively. Then we say that $\mathcal{T}$ is the \emph{iteration tree induced by} $j$. 
\end{definition}

The next lemma makes it clear how we would like to combine Theorem \ref{VM} and Lemma \ref{lemma}.
\begin{lemma} \label{InducedIteration}
	Assume $(\Delta)$. Suppose that $j:V \rightarrow M$ and $M^{\omega} \subseteq M$. Then there is $\Tt$ an iteration tree on $L[E]$ induced by $j$ such that: \begin{enumerate}
		\item[(a)] There is no drop along the main branch $b^{\Tt}$ of $\Tt$.
		\item[(b)] $\pi^{\Tt}_{0,\infty}:\cM_{0}^{\Tt} \longrightarrow \cM_{\infty}^{\Tt}$, i.e., $dom(\pi_{0,\infty}^{\Tt})=\cM_{0}^{\Tt}$.
	    \item[(c)] $\pi^{\Tt}_{0,\infty}=j$.
	    \item[(d)] $\cM_{\infty}^{\Tt}=M$.
	\end{enumerate} 
\end{lemma}
\begin{proof} Apply Theorem \ref{VM} to $j$ and $M$ and let $\Tt$ be the iteration tree on $\cK$ induced by $j$. Apply lemma \ref{lemma} to obtain $\cK= V$ and hence $\pi^{\Tt}_{0,\infty} = j$. Thus $\Tt$ and $\pi^{\Tt}_{0,\infty}$ are as sought. 
	
	Let us verify (d).  Let $\psi_{\cK}(x)$ be as in Theorem \ref{Knm} such that $x \in \cK$ if and only if $\psi_{\cK}(x)$. By Lemma \ref{lemma} we have $(\forall x ~ \psi_{\cK}(x))^{L[E]}$, then by the elementarity of $j$ we have  $(\forall  x ~ \psi_{\cK}(x))^{M}$. Therefore $M = \cK^{M} = \mathcal{M}_{\infty}^{\mathcal{T}} $, where we get the last equality by Theorem \ref{VM}. 
\end{proof}

\begin{lemma} \label{LemmaFiniteTree}
	Let $\cM$ be a premouse and $\Tt$ be an iteration tree on $\cM$. If $lh(\Tt)\geq \omega + 1 $, then $(\cM_{\infty}^{\Tt})^{\omega} \not\subseteq \cM_{\infty}^{\Tt}$. 
\end{lemma}
\begin{proof}
Suppose $lh(\mathcal{T}) \geq \omega + 1 $ and  $b=[0,\omega]_{T}$ is the cofinal branch on $\omega$.  Let $\langle\kappa_{n} \mid n \in \omega \cap b \rangle $ be such that \begin{gather*} \forall n \in (b \setminus n_{0}) \ \Big(\kappa_{n} = crit( \pi_{n,\omega}^{\mathcal{T}}) \Big),  
	\end{gather*} where $n_{0}$ is large enough such that $D^{\Tt} \cap (n_{0},\omega)_{T}=\emptyset $, and for $n \in b \cap n_{0}$ we set $\kappa_{n} = \emptyset$. Denote $\vec{\kappa}:=\langle\kappa_{n} \mid n \in \omega \cap b \rangle $. 
	
	\begin{claim} \label{SeqCrit} $\langle\kappa_{n} \mid n \in \omega \cap b \rangle \not\in \mathcal{M}_{\omega}^{\mathcal{T}}$. \end{claim}
	\begin{proof}  For a contradiction suppose $\vec{\kappa} \in \mathcal{M}_{\omega}^{\mathcal{T}}$, and let  $m \in \omega \cap b$ and $\overline{x} \in \mathcal{M}^{\mathcal{T}}_{m}$ such that 
	\begin{gather*} \pi^{\mathcal{T}}_{m,\omega}(\overline{x}) = \vec{\kappa}.
	\end{gather*}  
	
	Then \begin{gather*} crit(\pi^{\mathcal{T}}_{m,\omega}) = \pi^{\mathcal{T}}_{m,\omega}(\overline{x})(m) \in ran(\pi_{m,\omega}^{\mathcal{T}}).
	\end{gather*}  
	This is a contradiction since $crit(\pi^{\mathcal{T}}_{m,\omega}) \not\in ran(\pi_{m,\omega}^{\Tt})$. Therefore $\vec{\kappa} \not\in \mathcal{M}_{\omega}^{\mathcal{T}}$.
	\end{proof}

The lemma will follow from our next claim:
\begin{claim}\label{MissingSequence}
	If $\vec{\kappa} \in \cM_{\infty}^{\Tt}$, then $\vec{\kappa} \in \cM_{\omega}^{\Tt}$.
\end{claim}
\begin{proof} If $lh(\Tt)=\omega+1$, then $M_{\infty}^{\Tt}=\cM_{\omega}^{\Tt}$ and there is nothing to do in this case. 
	
	Let us assume $\lh(\Tt)>\omega+1$. 	The normality of $\mathcal{T}$ implies the following:
	\begin{gather}\label{Sup}sup_{n\in \omega} \kappa_{n} \leq \sup \{ \nu_{m}^{\mathcal{T}} \mid m+1 \in (\omega \cap b)\} \leq \nu^{\mathcal{T}}_{\omega}.\end{gather}
	By Fact \ref{fact1}, it follows that 
	$ \nu_{\omega}^{\mathcal{T}}$ is a successor cardinal in $\mathcal{M}_{\omega+1}^{\mathcal{T}}$. Since $ \sup \{ \nu_{m}^{\mathcal{T}} \mid m+1 \in (\omega \cap b)\}$ is clearly a limit cardinal in $\mathcal{M}^{\mathcal{T}}_{\omega+1}$, we have that the second inequality in \eqref{Sup} is in fact a strict inequality. 
	
	Write $\nu^{*}:=  \sup \{ \nu_{m}^{\mathcal{T}} \mid m+1 \in (\omega \cap b)\}$, then 
	\begin{gather*} \label{Ineq} \{ \kappa_{n} \mid n \in \omega\} \subseteq sup_{n\in\omega}\kappa_{n} < (\nu^{*})^{+\mathcal{M}^{\mathcal{T}}_{\omega+1}} \leq \nu_{\omega}^{\mathcal{T}}. \end{gather*}
	
	In particular, \begin{gather*} \mathcal{M}_{\omega + 1}^{\mathcal{T}}|\nu_{\omega+1}^{\mathcal{T}} \models  ``sup_{n\in\omega}\kappa_{n} < (\nu^{*})^{+} \leq \nu_{\omega}^{\mathcal{T}}". 
	\end{gather*} 
	Since 
	\begin{gather*}\label{findseq2} \mathcal{M}_{\omega+1}^{\mathcal{T}}|\nu_{\omega+1}^{\mathcal{T}} = \mathcal{M}^{\Tt}_{\infty}| \nu_{\omega+1}^{\mathcal{T}},
	\end{gather*} 
	we have 
	\begin{gather}  \label{findseq} \mathcal{M}_{\infty}^{\mathcal{T}}|\nu_{\omega+1}^{\mathcal{T}} \models  ``sup_{n\in\omega}\kappa_{n} < (\nu^{*})^{+} \leq \nu_{\omega}^{\mathcal{T}}". 
	\end{gather} 
	 By \eqref{findseq} and $\vec{\kappa}\in \mathcal{M}^{\mathcal{T}}_{\infty}$, we have $\vec{\kappa}\in \mathcal{M}^{\mathcal{T}}_{\infty}|\nu_{\omega}^{\mathcal{T}}$.
	  
	We also have 
	\begin{gather*} \mathcal{M}_{\omega+1}^{\mathcal{T}}|\nu_{\omega}^{\mathcal{T}} = \mathcal{M}^{\Tt}_{\infty}| \nu_{\omega}^{\mathcal{T}} = \mathcal{M}_{\omega}^{\mathcal{T}}|\nu_{\omega}^{\mathcal{T}},
	\end{gather*} 
	and hence \begin{gather*}
\vec{\kappa} \in \mathcal{M}_{\omega}^{\mathcal{T}}.  
	\end{gather*}  
\end{proof}
Hence if $(\cM_{\infty}^{\Tt})^{\omega} \subseteq M_{\infty}^{\Tt}$ and $lh(\Tt)\geq \omega+1$, by Claim \ref{MissingSequence} it follows that $\vec{\kappa}\in \cM_{\omega}^{\Tt}$ which by Claim \ref{SeqCrit} is a contradiction.
	\end{proof}

\begin{lemma}\label{finite} Assume $(\Delta)$. If $\kappa < \alpha$ are ordinals,  $ j:L[E] \longrightarrow M$  witnesses that $\kappa$ is $\alpha$-tall, and $\mathcal{T}$ is the iteration tree induced by $j$, then $ \mathcal{T}$ is finite. 
\end{lemma}

\begin{proof} Let $\Tt$ be the iteration tree induced by $j$. Lemma \ref{InducedIteration} implies that 
	$\pi^{\Tt}_{0,\infty}=j$ and $\cM_{\infty}^{\Tt}= M$.  We have that $M^{\omega} \subseteq M$ as $j:V \rightarrow M$ witnesses that $\kappa$ is $\alpha$-tall. Hence, by Lemma \ref{LemmaFiniteTree}, $\Tt$ is finite.
  \end{proof}

\begin{lemma} \label{lowcp3}
 Assume $(\Delta)$. Let $\kappa$ be a cardinal. Suppose $j:V \rightarrow M$ is an elementary embedding such that $\crit(j)= \kappa$ and $M^{\kappa}\subseteq M$. If $\Tt$ is the iteration tree induced by $j$, then for  $\alpha \in \lh(\Tt)$ we have that $\crit(E_{\alpha}^{\Tt}) \geq \kappa$. 
\end{lemma}
\begin{proof} Suppose for a contradiction that there is an $\alpha \in \lh(\Tt)$ such that $\crit(E_{\alpha}^{\Tt})< \kappa$.     Let $\pi_{E_{\alpha}^{\Tt}}:\cM_{\alpha}^{\Tt}||\nu_{\alpha}^{\Tt}\rightarrow \Ult_{0}(\cM_{\alpha}^{\Tt}||\nu_{\alpha}^{\Tt},E_{\alpha}^{\Tt})$ be the ultrapower map and $\tau_{\alpha}^{\Tt}=\crit(E_{\alpha}^{\Tt})^{+\cM_{\alpha}^{\Tt}||\nu_{\alpha}^{\Tt}}$. Since $\cM_{\alpha}^{\Tt}$ is a premouse, it follows that $\ran(\pi_{E_{\alpha}^{\Tt}}\restriction \tau_{\alpha})$ is cofinal in $\nu_{\alpha}^{\Tt}$.
   
    Notice that $\tau_{\alpha}^{\Tt}\leq \kappa$. Hence by $M^{\kappa}\subseteq M$ it follows that $$M \models ``\cf(\nu_{\alpha}^{\Tt}) \leq \crit(E_{\alpha})^{+\cM_{\alpha}^{\Tt}||\nu_{\alpha}^{\Tt}} < \nu_{\alpha}^{\Tt}".$$
    
   On the other hand by Lemma \ref{InducedIteration} we have $\cM_{\infty}^{\Tt} = M$ and by Fact \ref{fact1},
    $$\cM_{\infty}^{\Tt}    \models ``\nu_{\alpha}^{\Tt} \text{ is a successor cardinal}".$$ 
    
    This is a contradiction. Therefore for all $\alpha \in \lh(\Tt)$ we have $\crit(E_{\alpha}^{\Tt})\geq \kappa$.
\end{proof}

We remind the reader that our definitions of $o(\kappa)$ and $O(\kappa)$ are different from the usual ones. For the usual definitions the following lemma would be false. 
\begin{lemma} \label{TallMeas} Assume $(\Delta)$. If $\kappa$ is a measurable cardinal\footnote{The fact that the first part of this lemma holds under much weaker hypothesis is due  Schlutzenberg, see \cite{FarmerThesis}. Here we are working with the hypothesis that $(\Delta)$ holds which makes it easy to verify the lemma.} , then $o(\kappa)>\kappa^{+}$. If $\mu$ is a cardinal, $cf(\mu)>\kappa$ and $\kappa \ \text{is} \ \mu\text{-strong}$, then $  o(\kappa) > \mu $.
\end{lemma}
\begin{proof}  Let $j:L[E] \rightarrow M$ witness either that $\kappa$ is measurable or that $\kappa$ is $\mu$-strong. Because $cf(\mu) > \kappa$ in both cases we have that $M^{\kappa} \subseteq M $.

	Let $\Tt$ be the iteration tree induced by $j$ and let $E_{\alpha}^{\Tt}$ be the first extender applied to $\cM_{0}^{\Tt}$ such that $\alpha+1 \in b^{\Tt}$, the main branch of $\Tt$. Since $ \crit(j)=\crit(\pi_{0,\infty}^{\mathcal{T}}) = \kappa$ and $\mathcal{T}$ is a normal iteration tree, it follows that $\crit(E_{\alpha}^{\Tt}) = \kappa$. 
	
		\begin{claim}\label{MuStrong} If $j$ witnesses that $\kappa$ is $\mu$-strong, then  $\nu_{0}^{\mathcal{T}} > \mu$.
		\end{claim}
		\begin{proof}
			Suppose not. Then $\nu_{0}^{\mathcal{T}} \leq \mu$, and since we do not index extenders on cardinals it follows that $\nu_{0} < \mu$.  Hence,
			\begin{gather*}    (L[E]|\mu ) \models ``cf(\nu_{0}) \leq  (crit(E_{0}^{\mathcal{T}})^{+L[E]|\nu_{0}}) < \nu_{0} < \mu". 
			\end{gather*}
			Since
			\begin{gather*} M|\mu= \mathcal{M}^{\mathcal{T}}_{\infty}|\mu \supseteq V_{\mu}^{L[E]} \supseteq L[E]|\mu,
			\end{gather*} 
			it follows that  
			\begin{gather*} 
			M \models ``cf(\nu_{0}) \leq  (crit(E_{0}^{\mathcal{T}})^{+L[E]|\nu_{0}}) < \nu_{0} < \mu". 
			\end{gather*}
			But $\mathcal{M}^{\mathcal{T}}_{\infty} = M$ and $ \mathcal{M}^{\mathcal{T}}_{\infty}  \models ``\nu^{\mathcal{T}}_{0} \ \text{is regular cardinal}"$, which is a contradiction. \end{proof}
	
	$\blacktriangleright$ Suppose $ \alpha = 0 $. We split our analysis into two cases depending on whether $j$ witnesses that $\kappa$ is a measurable cardinal or $j$ witnesses that $\kappa$ is $\mu$-strong. 
	
	Suppose $j$ witnesses that  $\kappa$ is a measurable cardinal. As $\alpha = 0$, it follows that $E_{0}^{\Tt}$ is a total measure over $L[E]$ which implies that $\nu_{0}^{\Tt} > \kappa^{+}$. Therefore $E_{0}^{\Tt}$ witnesses that $O(\kappa)>\kappa^{+}$. By Lemma \ref{Oo} we have $o(\kappa)>\kappa^{+}$.
	
	Suppose $j$ witnesses that $\kappa$ is $\mu$-strong. In this case, by Claim \ref{MuStrong} it follows that $E_{0}^{\Tt}$ witnesses $O(\kappa) > \mu$. Therefore by Lemma \ref{Oo} it follows that $o(\kappa) > \mu$. 
	
	$\blacktriangleright$ Suppose we are in the case where $\alpha > 0$. 
		
	\begin{claim}\label{Projectum} If $j$ witnesses that $\kappa$ is measurable, 		 then for every $\gamma \geq \nu_{\alpha}^{\Tt }$ we have $\rho_{1}(\cM_{\alpha}^{\Tt}||\gamma) > \kappa^{+}$ . If $j$ witnesses that $\kappa$ is $\mu$-strong, then for every $\gamma \geq \nu_{\alpha}^{\Tt }$ we have $\rho_{1}(\cM_{\alpha}^{\Tt}||\gamma) > \mu$. 
	\end{claim} 
	\begin{proof}	
		 Suppose $j$ witnesses that $\kappa$ is a measurable cardinal. By Lemma \ref{lowcp3} it follows that $\kappa \leq \crit(E_{0}^{\Tt}) < \nu_{0}^{\Tt}$.  As  $\cM_{0}^{\Tt}|\nu_{0}^{\Tt} = \cM_{\alpha}^{\Tt}|\nu_{0}^{\Tt}$ it follows that $\kappa^{+\cM_{0}^{\Tt}|\nu_{0}^{\Tt}} = \kappa^{+\cM_{\alpha}^{\Tt}} < \nu_{0}^{\Tt}$. Since  $\crit(E_{\alpha}^{\Tt})=\kappa$, $\alpha+1 \not\in D^{\Tt}$ and  $E_{\alpha}^{\Tt}$ is the first extender used along the main branch of $\Tt$, altogether we have that $\kappa^{+}=\kappa^{+\cM_{\alpha}}$. Hence $\kappa^{+} <\nu_{0}^{\Tt}$.
		
		As every initial segment of $\cM_{0}^{\Tt}$ is sound, it follows that  $\rho_{1}(\cM_{0}^{\Tt}|\gamma) \geq \kappa^{+}$ for any $\gamma \geq \nu_{0}^{\Tt}$. Hence one can verify by  induction along the branch $(0,\alpha]_{T}$ that $\rho_{1}(\cM_{\beta}^{\Tt}||\gamma) > \kappa^{+} $ for any $\beta \in (0,\alpha]_{T}$ and any $\gamma \geq \nu_{0}^{\Tt}$. 
		
		Next suppose $j$ witnesses that $\kappa$ is a $\mu$-strong cardinal. Then $\nu_{\alpha}^{\Tt} > \nu_{0}^{\Tt} > \mu$, where the second inequality is Claim \ref{MuStrong} and we have $\rho_{1}(\cM_{0}^{\Tt}||\gamma) \geq \mu$ for any $\gamma \geq \nu_{0}^{\Tt}$.  Again, by induction along $(0,\alpha]_{T}$ one can verify that $\rho_{1}(\cM_{\beta}^{\Tt}||\gamma) > \mu$ for any $\beta \in (0,\alpha]_{T}$ and any $\gamma \geq \nu_{0}^{\Tt}$. 
	\end{proof} 
	 
	 We again split our analysis depending on whether $j$ witnesses that $\kappa$ is a measurable cardinal or $j$ witnesses that $\kappa$ is $\mu$-strong. 
	 
	 Suppose $j$ witnesses that $\kappa$ is a measurable cardinal. By Claim \ref{Projectum} we have $\rho_{1}(\cM_{\alpha}^{\Tt}||\nu_{\alpha}^{\Tt}) \geq \kappa^{++\cM_{\alpha}||\nu_{\alpha}^{\Tt}}$, hence we can apply Lemma \ref{IS} to $\cM_{\alpha}^{\Tt}||\nu_{\alpha}^{\Tt}$ and $E_{\alpha}^{\Tt}$ to find $E_{\gamma}^{\mathcal{M}_{\alpha}^{\mathcal{T}}}$ with $crit(E_{\gamma}^{\mathcal{M}_{\alpha}^{\mathcal{T}}||\nu_{\alpha}^{\Tt}}) = \kappa $ and $ \gamma \in (\kappa^{+},\kappa^{++\cM_{\alpha}^{\Tt}||\nu_{\alpha}^{\Tt}}) $. As $\nu_{0}^{\Tt} > \kappa^{++\cM_{\alpha}^{\Tt}||\nu_{\alpha}^{\Tt} }$ it follows that $E_{\gamma}^{\cM_{\alpha}^{\Tt}}=E_{\gamma}$. Therefore $O(\kappa) > \kappa^{+}$ and by Lemma \ref{Oo} we have $o(\kappa)>\kappa^{+}$.
	
	  Suppose $j$ witnesses that $\kappa$ is a $\mu$-strong cardinal.  By Claim \ref{Projectum} we have $\rho_{1}(\cM_{\alpha}^{\Tt}||\nu_{\alpha}^{\Tt}) \geq \mu^{+\cM_{\alpha}||\nu_{\alpha}^{\Tt}}$, hence we can apply Lemma \ref{IS} to $\cM_{\alpha}^{\Tt}||\nu_{\alpha}^{\Tt}$ and $E_{\alpha}^{\Tt}$ to find $E_{\gamma}^{\mathcal{M}_{\alpha}^{\mathcal{T}}}$ with $crit(E_{\gamma}^{\mathcal{M}_{\alpha}^{\mathcal{T}}||\nu_{\alpha}^{\Tt}}) = \kappa $ and $ \gamma \in (\kappa^{+L[E]},\mu^{+\cM_{\alpha}^{\Tt}||\nu_{\alpha}^{\Tt}}) $. As $\nu_{0}^{\Tt} \geq \mu^{+\cM_{\alpha}^{\Tt}||\nu_{\alpha}^{\Tt} }$ it follows that $E_{\gamma}^{\cM_{\alpha}^{\Tt}}=E_{\gamma}$. Thus $O(\kappa)>\mu$ which by Lemma \ref{Oo} implies $o(\kappa) > \mu$. \end{proof}

We now have all the technical tools we need to prove Theorem A. We shall use the following two results due to Hamkins which establish one direction of Theorem A. 

\begin{thm}[Theorem 2.10 in  \cite{tall}] \label{Hamkins} Suppose $V$ is an extender model $L[E]$ that is normaly iterable. If $\mu$ is a cardinal, $cf(\mu) > \kappa$  and $ o(\kappa) > \mu $ then $ \kappa \ \text{is} \ \mu\text{-tall} $. 
\end{thm}

\begin{thm}[Corollary 2.7 in  \cite{tall}] \label{PropI} Suppose $V$ is an extender model $L[E]$ that is normaly iterable.  If $o(\kappa) >\kappa^{+}$ and $\sup \{\alpha < \kappa  \mid  o(\alpha) > \mu \} = \kappa$  then $ \kappa$ is $\mu$-tall. 
	
	Moreover, if $ o(\kappa)>\kappa^{+} $ and $$\sup \{\alpha < \kappa  \mid  \alpha \ \text{is a strong cardinal} \} = \kappa,$$ then
	$ \kappa$ is a tall cardinal.
\end{thm}

We now prove the main theorem. 
\begin{thma}\label{MainThm} Assume $(\Delta)$.  Let $\kappa <  \mu$ be regular cardinals. Suppose further that $L[E]|\mu$ is $\mu$-stable above $\kappa$. Then $  \kappa \ \text{is} \ \mu\text{-tall} $ iff \begin{gather*} o(\kappa) > \mu \\ \text{ or} \\ \Big( o(\kappa) > \kappa^{+} \wedge  sup\{ \nu < \kappa \mid  o(\nu) > \mu \} = \kappa  \Big) \end{gather*}
\end{thma}


\begin{proof}
($\Leftarrow$) It follows from  \ref{Hamkins} and \ref{PropI}.

($\Rightarrow$) Since $\kappa $ is $\mu$-tall, we have that $\kappa$ is measurable and hence by Lemma \ref{TallMeas} we also have that $o(\kappa) > \kappa^{+}$.

 Suppose that  $ \kappa > sup(\{ \alpha < \kappa \mid  o(\alpha) > \mu \}) $. We will verify that $o(\kappa) > \mu$. Towards a contradiction, suppose that $o(\kappa) \leq \mu$. 
 
   By \ref{Oo}, $  sup(\{\beta< \kappa \mid  o(\beta) > \mu  \} ) < \kappa $  implies that $sup(\{\beta < \kappa \mid O(\beta) > \mu\}) < \kappa$. Let  
\begin{gather*} \beta^{*}:= sup(\{\beta< \kappa \mid  O(\beta) > \mu  \} ), 
\end{gather*}  and set 
\begin{gather*} \Theta:= sup(\{O(\beta) \mid \beta^{*}< \beta \leq \kappa \}).   
\end{gather*}

Notice that $\Theta \leq \mu$: by Lemma \ref{Oo}, $o(\kappa) \leq \mu $ implies $O(\kappa) \leq \mu$ and the definition of $\beta^{*}$ implies that  $sup(\{O(\beta) \mid \beta^{*}< \beta < \kappa \}) \leq \mu$.

  Let $ j:V \rightarrow M $ witnesses the $ \mu$-tallness of $\kappa$ and let $ \mathcal{T}$ be the iteration tree induced by $j$. Lemma \ref{InducedIteration} gives that $j= \pi^{\Tt}_{0,\infty}$ and $\mathcal{M}_{\infty}^{\mathcal{T}} = \cK^{M} = M$  and Lemma \ref{finite} implies  that $\Tt$ is finite.

Let $b$ be the main branch of $\mathcal{T}$. We know by Theorem \ref{VM} (or Lemma \ref{InducedIteration}) that there is no drop along $b$ so we can define
\begin{gather*} t_{0} = min  \Big(\{ m \in b \ |  \  \pi_{0,m}^{\Tt}(\kappa)=\pi_{0,\infty}^{\Tt}(\kappa) \} \Big).
\end{gather*}

For $\Upsilon \in \ord$ and $k \in lh(\Tt) $  let $\psi(k,\Upsilon)$ denote the following statement:
$$\cM_{k}^{\Tt} \models ``\forall \zeta > \Upsilon ~  ( crit(E_{\zeta}) \not\in (\mu,\Upsilon))".$$

\begin{claim}\label{bound} Let $n \leq t_{0}$. Then $\nu^{\mathcal{T}}_{n} <\mu$ and whenever $\pi^{\Tt}_{0,n}(\Theta)$ is defined we have $ \nu^{\Tt}_{n} \leq \pi^{\Tt}_{0,n}(\Theta) \leq \mu$ .
\end{claim}

Before we start with the proof of Claim \ref{bound} we observe that Lemma \ref{Lives} and  Claim \ref{bound} imply that $\Tt$ lives on $L[E]|\mu$. Our hypothesis that $L[E]| \mu$ is $\mu$-stable above $\kappa$ together with Lemma \ref{StablePM} imply that $\pi_{0,\infty}^{\Tt}(\kappa) \leq \mu$. This give us a contradiction since $\mu < j(\kappa)=\pi_{0,\infty}^{\Tt}(\kappa)$ and therefore we have  $o(\kappa)>\mu$. Hence Theorem A will follow once we prove Claim \ref{bound}.

\begin{proof}[Proof of Claim \ref{bound}]
We prove this by induction on $n \leq t_{0}$. We start with the base case.
 
 \begin{subclaim} \label{BaseStep} For $n=0$, we have that $\nu^{\Tt}_{0} \leq \Theta$ and $\nu^{\Tt}_{0} < \mu$.
 	\end{subclaim} 
 	\begin{proof} For a contradiction suppose that $\nu^{\Tt}_{0} > \Theta $. We will prove that for all $k + 1 < \lh(\Tt)$ we have $\crit(E_{k}^{\Tt})> \Theta$. This will imply that if $E_{k_{0}}^{\Tt}$ is the first extender applied to $\cM_{0}^{\Tt}$ with $k_{0}+1 \in [0,t_{0}]_{T}$, then 
 		$$ \crit(\pi^{\Tt}_{0,\infty}) = \crit(E_{k_{0}}^{\Tt}) > \Theta > \kappa = \crit(\pi^{\Tt}_{0,\infty}),$$ which will give us a contradiction.

  Suppose  $\psi(k,\Theta)$ holds. That is, 
  $$\cM_{k}^{\Tt} \models ``\forall \zeta > \Theta ~  ( crit(E_{\zeta}) \not\in (\mu,\Theta))".$$
  As $\Tt$ is normal we have $\nu_{k}^{\Tt} \geq \nu_{0}^{\Tt} >\Theta$, and therefore, $\crit(E_{k}^{\Tt}) \leq \beta^{*}$ or $\crit(E_{k}^{\Tt}) > \Theta$. By Lemma \ref{lowcp3}, it follows that $\crit(E_{k}^{\Tt}) >\Theta$. Thus given $k\in \lh(\Tt)$, $\psi(k,\Theta)$ implies $\crit(E_{k}^{\Tt})>\Theta$. 
 
We will verify by induction that $ \psi(k,\Theta)$ holds for all $k \in lh(\Tt)$. For $k=0$, Lemma \ref{gap} implies $ \psi(0,\Theta)$. Suppose now that $k \in \lh(\Tt)$ and $\psi(l,\Theta)$ holds for all $l \leq k$. We will prove that  $\psi(k+1,\Theta)$ holds.  

As observed above, $ \psi(k,\Theta)$ implies $\crit(E_{k}^{\Tt}) >\Theta$.  By the induction hypothesis, $ \psi(k,\Theta)$ holds, and hence we have\footnote{Recall $\eta^{\Tt}_{k+1}$ is the height of the model we apply $E_{k}^{\Tt}$ to form $\cM_{k+1}^{\Tt}$.} 
$$\eta^{\mathcal{T}}_{k} > crit(E_{k}^{\mathcal{T}}) > \Theta.$$ By induction hypothesis we also have   $\psi(\xi^{\Tt}_{k},\Theta)$, therefore the following holds:
 \begin{gather*}  (\mathcal{M}^{\mathcal{T}}_{\xi_{k}})||\eta_{k}^{\Tt}  \models ~ ``\forall  \gamma > \Theta ~ ( crit(E_{\gamma}) \not\in (\mu,\Theta) )". 
\end{gather*} 

 Thus by the $\Sigma_{1}$-elementarity of $\pi_{\xi_{k+1},k}^{\Tt}$ it follows that $\psi(k+1,\Theta)$ holds. This concludes the proof that $\psi(k,\Theta)$ holds for all $k \in \lh(\Tt)$.

 As observed above, this implies that $\crit(\pi_{0,\infty}^{\Tt}) > \Theta$, which is a contradiction.
 
 Hence  \begin{gather*}\label{case0} \nu^{\mathcal{T}}_{0} \leq \Theta \leq \mu,
 \end{gather*} and as $\mu$ is a cardinal we also have $\nu_{0}^{\Tt} <\mu$. This concludes the case $n=0$ of Claim \ref{bound} and the proof of subclaim \ref{BaseStep}.  \end{proof}

We now perform the inductive step of the proof of Claim \ref{bound}. We shall need to split this into two cases. 

$\blacktriangleright$ Suppose $n=k+1$ and $\pi_{0,k+1}^{\Tt}(\Theta)$ is not defined. 
By our induction hypothesis  $\nu_{l}^{\Tt} <\mu$  for all $l < k+1$, hence by Lemma \ref{Lives}, the iteration tree  $\mathcal{T}|(k+2)$ lives on $L[E]|\mu$. Therefore by Lemma \ref{IterationBounds},
\begin{gather}\label{BoundThetaI}  \cM^{\Tt}_{k+1}=\cM^{\Uu}_{k+1}
\end{gather}  where $\Uu=(\Tt\restriction k+2)\restriction (L[E]|\mu)$. By Lemma \ref{StablePM} we have $\cM^{\Uu}_{k+1}\cap \ord \leq \mu$. Therefore $\nu^{\Tt}_{k+1} <\mu$. 
 
$\blacktriangleright$ Suppose $ n=k+1$ and $\pi_{0,k+1}^{\Tt}(\Theta)$ is defined.
By our induction hypothesis  $\nu_{l}^{\Tt} <\mu$  for all $l < k+1$, hence by Lemma \ref{Lives} the iteration tree  $\mathcal{T}\restriction (k+2)$ lives on $L[E]|\mu$. Therefore by Lemma \ref{IterationBounds}
\begin{gather}\label{BoundTheta}  \pi_{0,k+1}^{\mathcal{T}}(\Theta) \leq \mu.
\end{gather}
Thus we only have to verify that $\nu_{k+1}^{\Tt} \leq \pi_{0,k+1}^{\Tt}(\Theta)$.
 From our induction hypothesis we have $\nu^{\mathcal{T}}_{\xi_{k}^{\Tt}} \leq \pi_{0,\xi_{k}^{\Tt}}^{\Tt}(\Theta)$. Let $\tau_{k}^{\Tt}=crit(E^{\mathcal{T}}_{k})^{+\cM^{\Tt}_{k+1}||\nu_{k}^{\Tt}}$, then 
 $$ \tau_{k}^{\Tt} < \nu^{\mathcal{T}}_{\xi_{k}^{\Tt}} \leq \pi^{\Tt}_{0,\xi_{k}^{\Tt}}(\Theta), $$
  which implies the following:
  \begin{gather*} \pi_{0,k+1}^{\Tt}(\Theta)   \geq   \pi^{\mathcal{T}}_{\xi_{k}^{\Tt},k+1}(\nu^{\mathcal{T}}_{\xi_{k}^{\Tt}}) > \pi^{\mathcal{T}}_{\xi_{k}^{\Tt},k+1}(\tau^{\mathcal{T}}_{k})=\nu^{\mathcal{T}}_{k}.
  \end{gather*}   Therefore, \begin{gather}\label{trapped}\nu^{\mathcal{T}}_{k} < \pi_{0,k+1}^{\Tt}(\Theta). \end{gather} 
 
 Suppose for a contradiction that $ \nu^{\mathcal{T}}_{k+1} > \pi_{0,k+1}^{\Tt}(\Theta)$. Using Lemma \ref{gap}, one can verify by induction\footnote{We use induction like in case $n=0$, there may be drops in model along $[0,k+1]_{\mathcal{T}}$ but by hypothesis $\pi_{0,k+1}^{\Tt}(\Theta) $ is defined for all $m \in [0,k+1]_{\mathcal{T}}$.} along $[0,k+1]_{\mathcal{T}}$ that $ \psi(k+1,\pi^{\Tt}_{0,k+1}(\Theta))$ holds. Recall that  $\psi(k+1,\pi^{\Tt}_{0,k+1}(\Theta))$ denotes
\begin{gather*}\mathcal{M}^{\mathcal{T}}_{k+1} \models ``\forall \nu > \pi_{0,k+1}^{\Tt}(\Theta) \ (\ crit(E_{\nu}) \not\in (\beta^{*}, \pi_{0,k+1}^{\Tt}(\Theta)))". \end{gather*}

\begin{subclaim} \label{claimtrapped}
	For $l \in \lh(\Tt)$, if  $k+1 < l$ then \begin{equation}\label{trappedI} 
	\begin{gathered}    ( k+1 <_{T} l ) \wedge  \psi(l,\pi_{0,k+1}^{\Tt}(\Theta)).
	\end{gathered}\end{equation}
\end{subclaim}
\begin{proof}We will verify this by induction, similar to what we did for the case $n=0$.

We shall start by verifying that $k+1 <_{T} k+2$. As $\nu_{k+1}^{\Tt} > \pi_{0,k+1}^{\Tt}(\Theta)$ and $\psi(k+1,\pi_{0,k+1}^{\Tt}(\Theta))$ holds, we must have, by \eqref{trapped}, that $\crit(E_{k+1}^{\Tt}) > \Theta > \nu^{\Tt}_{k}$. Thus $\xi_{k+1}^{\Tt} = k+1$ and $k+1 <_{T} k+2$. 

Next we verify that  $\psi(k+2,\pi_{0,k+1}^{\Tt}(\Theta))$ holds. As $\psi(k+1,\pi_{0,k+1}^{\Tt}(\Theta))$ holds, it follows that  \begin{gather*}
(\mathcal{M}_{\xi_{k+1}^{\Tt}})^{\mathcal{T}}||\eta^{\Tt}_{k+1} \models  ``\forall  \zeta > \pi^{\Tt}_{0,k+1}(\Theta) \  ( crit(E_{\zeta}) \not\in (\beta^{*},\pi_{0,k+1}^{\Tt}(\Theta)) )".  
\end{gather*} Then by the $\Sigma_{1}$-elementarity of $\pi_{k+1,k+2}^{\Tt}$ it follows that $\psi(k+2,\pi^{\Tt}_{0,k+1}(\Theta))$ holds. This concludes the base step of this induction.

We now prove the inductive step. Suppose \eqref{trappedI} holds for all $l$ such that  $k+1 < l \leq m $, let us verify that \eqref{trappedI} holds for $m+1$. As $\Tt$ is normal, we have $\nu^{\mathcal{T}}_{m} > \nu^{\mathcal{T}}_{k}$. By our induction hypothesis $\psi(m,\pi_{0,k+1}^{\Tt}(\Theta))$ holds and we have only two possibilities: $crit(E^{\mathcal{T}}_{m}) \leq \beta^{*} $ or $crit(E^{\mathcal{T}}_{m}) > \pi_{0,k+1}^{\Tt}(\Theta)$. 

The first possibility is excluded by Lemma \ref{lowcp3}, and so it must be true that $crit(E^{\mathcal{T}}_{m}) > \pi_{0,k+1}^{\Tt}(\Theta)$. 
By \eqref{trapped}, $ crit(E^{\mathcal{T}}_{m}) > \pi_{0,k+1}^{\Tt}(\Theta)$ implies $crit(E^{\mathcal{T}}_{m}) > \nu^{\mathcal{T}}_{k} $. Therefore $ k+1 \leq \xi_{m}^{\Tt} = \pred_{T}((m+1))$ and thus we have either $k+1 = \xi_{m}^{\Tt}$ or $k+1 < \xi_{m}^{\Tt}$. 

If the former is true, then we can appeal to $\psi(k+1,\pi_{0,k+1}^{\Tt}(\Theta))$,  and if the latter is true, then we can appeal to the induction hypothesis, and in either case, we can conclude that $\psi(\xi_{m}^{\Tt},\pi_{0,k+1}(\Theta))$ holds and 
\begin{gather*}
(\mathcal{M}_{\xi_{m}^{\Tt}})^{\mathcal{T}}||\eta^{\Tt}_{m} \models  ``\forall  \zeta > \pi^{\Tt}_{0,k+1}(\Theta) \  ( crit(E)_{\zeta}) \not\in (\beta^{*},\pi_{0,k+1}^{\Tt}(\Theta)) )".  
\end{gather*}

By $\Sigma_{1}$-elementarity  of $\pi_{\xi_{m}^{\Tt},m+1}^{\Tt}$, we have $\psi(m+1,\pi_{0,k+1}^{\Tt}(\Theta))$. Now if $k+1=\xi_{m}^{\Tt}$, then by the definition of $\xi_m^{\Tt}$, we have $k+1 <_{T} m+1$. On the other hand, if $k+1 < \xi_{m}^{\Tt}$, then by the induction hypothesis, $ k+1 <_{T} \xi_{m}^{\Tt} $, and so again we have that $k+1 <_{T} m+1$. This concludes the inductive step and the induction and verifies Subclaim \ref{claimtrapped}. 
\end{proof}

Given $l \in \lh(\Tt)$ such that $l \geq k +1$, we have that  $\psi(l,\pi_{0,k+1}^{\Tt}(\Theta))$ implies $\crit(E_{l}^{\Tt}) > \pi_{0,k+1}^{\Tt}(\Theta)$. Therefore  Subclaim~\eqref{claimtrapped} gives that $$\pi_{k+1,\infty}^{\mathcal{T}}\restriction (\pi^{\Tt}_{0,k+1}(\Theta)+1) = id\restriction (\pi_{0,k+1}^{\Tt}(\Theta)+1).$$ Hence  \begin{gather*} \pi_{0,\infty}^{\mathcal{T}}(\kappa) = \pi_{k+1,\infty} ^{\mathcal{T}}\circ \pi^{\mathcal{T}}_{0,k+1}(\kappa) =  id \circ \pi_{0,k+1}^{\mathcal{T}}(\kappa)\leq  \pi^{\Tt}_{0,k+1}(\Theta) \leq \mu, 
\end{gather*} where the last inequality is given by   \eqref{BoundTheta}. This contradicts the fact that $$\pi_{0,\infty}^{\Tt}(\kappa)=j(\kappa) >\mu.$$  Thus $\nu^{\mathcal{T}}_{k+1} \leq \pi_{0,k+1}^{\Tt}(\Theta) \leq \mu$. This verifies the case $n=k+1$. \end{proof}

As observed before the proof of Claim \ref{bound}, we have by Lemma \ref{Lives} and Claim \ref{bound} that $ \mu < j(\kappa) = \pi_{0,\infty}^{\mathcal{T}}(\kappa) \leq \mu$ which is a contradiction. Thus $o(\kappa) > \mu$. This concludes the proof of Theorem A.  
\end{proof}

\begin{definition}(Hamkins) A cardinal $\kappa$ is \emph{$<\alpha$-tall} if and only if for all $\beta < \alpha \ \ \kappa $ is  $\beta$-tall.
\end{definition}
\begin{corollary} \label{maincor} Assume $(\Delta)$.  Suppose that $ \alpha$ is a limit cardinal and $cf(\alpha)> \kappa$. Then $ \kappa \ \text{is} \ <\alpha\text{-tall}$ iff \begin{gather*} ``o(\kappa) \geq  \alpha   \\ \text{or}   \\   \ o(\kappa) > \kappa^{+} \wedge \kappa = sup\{ \beta < \kappa \mid o(\beta) \geq \alpha \}"  
\end{gather*}
\end{corollary}
\begin{proof}$(\Leftarrow)$ It follows from \ref{Hamkins} and \ref{PropI}.
	
$(\Rightarrow)$ Let $\langle\mu_{\xi} \mid \xi < cf(\alpha) \rangle $ be a cofinal sequence in $\alpha$. Note that for $\mu:= \mu^{++}_{\xi}$ we are in hypothesis of Theorem A. If for each $\mu^{++}_{\xi} > \kappa$  we have $o(\kappa) > \mu^{++}_{\xi}$ then $o(\kappa) \geq \alpha$ and we are done.  

Suppose $o(\kappa) <\alpha$. By Lemma \ref{TallMeas}, we have $o(\kappa)>\kappa^{+}$,  therefore we will be done if we find  $B \subseteq \kappa$ such that $B$ is cofinal in $\kappa$ and for all $\beta \in B $ we have $o(\beta) \geq \alpha$. 

Fix $\xi < cf(\alpha)$ such that $\mu^{++}_{\xi} > o(\kappa)$. As $\kappa$ is $\alpha$-tall, it follows that $\kappa$ is $\mu^{++}_{\xi}$-tall. Applying Theorem A to $\kappa$ and  $\mu_{\xi}^{++}$, as $o(\kappa) <\mu_{\xi}^{++}$, this gives us a set $Y \subseteq \kappa$ cofinal in $\kappa$ such that for all $ \beta \in Y$ we have $o(\beta) > \mu_{\xi}^{++}$.   

For each $ \xi < cf(\alpha)$, let 
\begin{gather*} B_{\xi} = \{ \beta < \kappa \mid  o(\beta) > \mu_{\xi}^{++} \}.
\end{gather*} 

Then $\langle B_{\xi} \mid \xi < cf(\alpha) \rangle $ is a sequence of cofinal subsets of $\kappa$ such that 
\begin{gather*} \forall \xi  \Big( \xi < \zeta < cf(\alpha) \ \longrightarrow  B_{\zeta} \subseteq B_{\xi} \Big).
\end{gather*} 

From the fact that $ cf(\alpha) > \kappa$, it follows that the sequence  $\langle B_{\xi} \mid \xi < cf(\alpha) \rangle $ is eventually constant, i.e., there is some $B \subseteq \kappa $ such that for all sufficiently large $\xi < cf(\alpha)$ we have $B=B_{\xi}$.  We have that $B$ is cofinal in $\kappa$ and for all $\beta \in B $ we have $o(\beta) \geq \alpha$, which verifies the corollary. 
\end{proof}

\begin{cora}Assume $(\Delta)$. $\kappa \ \text{is a tall cardinal } $ if and only if \begin{gather*}  \kappa \ \text{is a strong cardinal or a measurable limit of strong cardinals.} \end{gather*} 
\end{cora}

We will need one further notion of iterability for Theorem~B. 

\begin{definition}
	We say that a premouse $\cM$ is \emph{weakly iterable} if every countable premouse $\bar{\cM}$ that elementarily embeds into $\cM$ is $(\omega_{1}+1,\omega_{1})$-iterable\footnote{See \cite[p.309]{MR1876087} for the definition of $(\omega_{1}+1,\omega_{1})$-iterable.}.
\end{definition}

The following lemma together with Theorem A imply Theorem B.

\begin{lemma}\label{Standard} Suppose there is no inner model with a Woodin cardinal. Let $L[E]$ be a proper class premouse that is weakly iterable. Then $L[E]$ is self-iterable and $L[E]\models (\Delta)$.
\end{lemma} 
\begin{proof}[Proof Sketch]
	Let $\Tt \in L[E]$ be an iteration  on $L[E]$ of limit lenght. By our hypothesis that there is no inner model with a Woodin cardinal it follows that\footnote{$\cM(\Tt)$ denotes the common part model, see \cite[Definition 6.9]{MR2768698}. For $\delta(\Tt)= \bigcup_{\alpha \in \lh(\Tt)} \nu_{\alpha}^{\Tt}$ and $\mathbb{E}:=\bigcup_{\alpha \in \lh(\Tt)} E^{\cM_{\alpha}^{\Tt}\restriction \nu_{\alpha}^{\Tt}}$, $\cM(\Tt):=J_{\delta}^{\mathbb{E}}$.} $$L[\cM(\Tt)] \models ``\delta(\Tt) \text{ is not a Woodin cardinal}".$$ 
	
	Let $\eta$ be the least ordinal such that $$(L[\cM(\Tt)]|\eta+\omega) \models ``\delta \text{ is not a Woodin cardinal}",$$ and set $\cQ(\Tt):=  L[\cM(\Tt)]||\eta$.  Let $X$ be a countable set such that $X \prec_{\Sigma_{\omega}} H_{\Theta}^{\cM}$, let $\bar{X} $ be the Mostowski collapse of $X$ and $\pi:\bar{X} \rightarrow X $ be the inverse of the Mostowski collapse. For each $w \in X$ we will denote by $\bar{w} $ the pre-image of $w$ under $\pi$. 
	
	Since $L[E]$ is weakly iterable, it follows that there exists $b$ a cofinal wellfounded branch of $\bar{\Tt}$ such that $\cM_{\infty}^{\bar{\Tt}} \triangleright \cQ(b,\bar{\Tt})=\overline{\cQ(\Tt)}$ and such $b$ is unique\footnote{See \cite[Definition 6.11]{MR2768698} for the definition of $\cQ(b,\bar{\Tt})$. }.
	
	Let $G$ be a $Col(\omega,\nu)$-generic over $\bar{X}$, where $(\nu=|\bar{\Tt}|)^{\bar{X}}$. As $\bar{X}[G] \prec_{\Sigma_{1}^{1}} V$, it follows that $b \in \bar{X}[G]$ and by homogeneity of $Col(\omega,\nu)$ it follows that $b \in X$. 
	
	Therefore $$ H_{\Omega}^{\cM}\models \exists c ~ ( c \text{ is  a cofinal well founded branch of } \Tt)$$
\end{proof}

\begin{thm}
Suppose there is no inner model with a Woodin cardinal and $L[E]$ is an extender model that is self-iterable. Let $\kappa$, $\mu$ be ordinals such that $\kappa < \mu$ and $\mu$ is a regular cardinal. 
If $L[E]|\mu$ is $\mu$-stable above $\kappa$, then $(\kappa \text{ is } \mu\text{-tall})^{L[E]}$ iff 
\begin{equation}
\begin{gathered} \label{EquivalenceEquationI} (o(\kappa)) > \mu)^{L[E]} \\ \text{or} \\ o(\kappa) > \kappa^{+} \wedge \kappa = \sup\{\nu < \kappa \mid o(\nu)>\mu\})^{L[E]}
\end{gathered}
\end{equation}
In particular, if $L[E]$ is weakly iterable, then $(\kappa \text{ is } \mu\text{-tall})^{L[E]}$ iff \eqref{EquivalenceEquationI} holds.
\end{thm}
\begin{proof} The first part follows from the fact that $L[E] \models (\Delta)$ and Theorem A. The second part follows from Lemma \ref{Standard} and Theorem A. 
	\end{proof}

\begin{cor} Suppose there is no inner model with a Woodin cardinal and $L[E]$ is an extender model that is self-iterable. Let $\kappa$ be an ordinal. Then $ (\kappa \text{ is }\text{tall})^{L[E]}$ iff \begin{equation}\begin{gathered}\label{EquivalenceEquationII} (\kappa \text{ is a strong cardidnal} \\ \text{or} \\ \kappa \text{ is a measurable limit of strong cardinals})^{L[E]}.\end{gathered}\end{equation}
	In particular, if $L[E]$ is weakly iterable, then $(\kappa \text{ is tall})^{L[E]}$ iff \eqref{EquivalenceEquationII} holds.
	\end{cor}

Next we prove that we can not remove the hypothesis that $L[E]|\mu$ is $\mu$-stable above $\kappa$ in Theorem~A. For that we will use the following lemma.

\begin{lemma}[Lemma 2.3 in \cite{tall}] \label{Salvador} If $\kappa < \theta$ are ordinals and $j:V \longrightarrow M $ is an elementary embedding such that $ ^{\kappa}M \subseteq M$ and $j(\kappa) \geq \theta$, then 
$ \kappa \ \text{ is} \ \theta$-tall.
\end{lemma}
\begin{proof}[Proof sketch] By elementarity of $j$ it follows that $j(\kappa)$ is measurable in $M$, let $U \in M$ be a total measure on $M$ with $crit(U) = j(\kappa)$ and $i$ the ultrapower embedding from $U$, then $i\circ j$ witness that $\kappa$ is $\theta$-tall.
\end{proof}

\begin{lemma}\label{CES}
	Suppose $(\Delta)$ .  Let $ \kappa $ and $\lambda$  be cardinals. If  \begin{enumerate}
		\item [(I)] \begin{enumerate}
			\item[1.] $\kappa < \lambda$, 
			\item[2.] $o(\kappa) \in [\lambda ,\lambda^{+}) $, 
			\item[3.] for some $\mu$, $\kappa < cf(\mu)$ and
			\item[4.] $\cf(\mu)$ is a measurable cardinal. 
		\end{enumerate} 
	\end{enumerate}	or	\begin{enumerate}
		\item[(II)] \begin{enumerate}
			\item[1.] $\kappa$ is a measurable cardinal, 
			\item[2.] $\lambda > \kappa$ and $\cf(\lambda) = \kappa$,
			\item[3.] $\kappa = \sup(\{\beta <\kappa \mid o(\beta) > \lambda \})$, 
		\end{enumerate} 
	\end{enumerate}
	
	Then $\kappa$ is $\lambda^{+}$-tall.
\end{lemma}

\begin{proof} $\blacktriangleright$ Suppose (I) holds. By Lemma \ref{TallMeas} there is  $E_{\beta}$ a total measure with $\crit(E_{\beta})=\cf(\lambda)$. Consider $\cN= \Ult(V,E_{\beta})$ and let $\pi$ be the ultrapower map.
	
	We have $\pi(\lambda) \geq \lambda^{+}$, which implies 
	$$\cN\models o(\kappa)  \in [\pi(\lambda),\pi(\lambda^{+V})] \subseteq [\lambda^{+V},\pi(\lambda)^{+\cN})$$
	
	Let $E_{\alpha}^{\cN}$ be such that $\alpha > \lambda^{+V}$ and $\crit(E_{\alpha}^{\cN})=\kappa$.
	Let $\cW:=Ult(\cN,E_{\alpha}^{\cN})$ and let $i:\cN \rightarrow \cW$ be the ultrapower map. 
	It follows that $i\circ \pi(\kappa) \geq \lambda^{+V}$ and $\cW^{\kappa}\subseteq \cW$. 
	Hence by Lemma \ref{Salvador} $\kappa$ is $\lambda^{+}$-tall. 
	
	$\blacktriangleright$ Suppose (II) holds. By Lemma \ref{TallMeas} there is $E_{\beta}$ a total measure with $\crit(E_{\beta}) = \kappa$. Consider $\cN= \Ult(V,E_{\beta})$ and let $\pi:V \rightarrow \cN$ be the ultrapower map. 
	
	We have \begin{gather} \label{equation3} \cN\models \sup(\{ \beta < \pi(\kappa) \mid o(\beta) \geq \pi(\lambda)\}) = \pi(\kappa)
	\end{gather}  
	
	Notice that $\pi(\lambda) > \lambda^{+V}$. By \eqref{equation3} and Lemma \ref{IS} there is $\gamma \in (\lambda^{+V},\pi(\lambda))$ such that $\crit(E^{\cN}_{\gamma}) \in (\kappa,\pi(\kappa))$.
	
	If we consider $F:= E_{\gamma}\restriction \lambda^{+}$, as $\cf(\lambda^{+V})= \lambda^{+V} > \kappa$ and  $\cN^{\kappa}\subseteq \cN$, it follows that $\cW:= \Ult(\cN,F)$ is $\kappa$-closed. 
	
	We also have for $i:\cN\rightarrow \cW$, the ultrapower map, that $i \circ \pi (\kappa) \geq \lambda^{+V}$.
	Hence by Lemma \ref{Salvador} $\kappa$ is $\lambda^{+}$-tall.
	
	\end{proof}

\begin{remark}
	Notice that in Lemma \ref{CES} setting $\mu=\lambda^{+}$ it follows that $L[E]| \mu$ is not $\mu$-stable above $\kappa$, as $\lambda$ is the largest cardinal of $L[E]|\mu$ and $cf(\lambda) $ is a measurable $> \kappa$.
\end{remark}

\section{Acknowledgments}

The authors express their gratitude to Tanmay Inamdar for reading earlier versions of this paper and providing many helpful comments and suggestions.
The authors express their gratitude to the referee for a thorough reading and valuable report.

\bibliographystyle{jflnat}

\end{document}